\documentclass[12pt]{amsart}
\usepackage{amsmath, amsthm, amscd, amsfonts, amssymb, graphicx, color,float,pgf,tikz}
\usepackage{amssymb,fontenc}
\usepackage{latexsym,wasysym,mathrsfs}
\usepackage{hyperref}
\usetikzlibrary{arrows}
\usepackage[square,numbers,sort&compress]{natbib}

\makeatletter \oddsidemargin.9375in \evensidemargin \oddsidemargin
\newtheorem{pro}{Proposition}[section]
 \newtheorem{cor}[pro]{Corollary}
 \newtheorem{lem}[pro]{Lemma}
 \newtheorem{rem}[pro]{Remark}
 \newtheorem{exa}[pro]{Example}
  \newtheorem{defi}[pro]{Definition}
 \newtheorem{prop}[pro]{Proposition}
 \newtheorem{thm}[pro]{Theorem}

 \newcommand{\h}{\mathcal H}

   \newcommand{\am}{\mathcal{A} }

 \numberwithin{equation}{section}

\begin{document}

\setcounter{page}{1}

\title[  Controlled $K$-Fusion Frame in  Hilbert C$^*$-modules]{Robustness of  controlled $K$-Fusion Frame in Hilbert  C$^*$-modules under erasures of submodules}

\author[NADIA ASSILA, SAMIR KABBAJ and MOHAMED OTMANI ]{NADIA ASSILA$^{1*}$, SAMIR KABBAJ $^2$ \MakeLowercase{and} MOHAMED OTMANI$^3$}

\address{$^{1}$ Department of Mathematics, University of Ibn Tofail, B.P. 133, Kenitra, Morocco}
\email{\textcolor[rgb]{0.00,0.00,0.84}{samkabbaj@yahoo.fr}}

\address{$^{2}$Department of Mathematics, University of Ibn Tofail, B.P. 133, Kenitra, Morocco}
\email{\textcolor[rgb]{0.00,0.00,0.84}{nadiyassila@gmail.com}}

\address{$^{1}$Department of Mathematics, University of Ibn Tofail, B.P. 133, Kenitra, Morocco}
\email{\textcolor[rgb]{0.00,0.00,0.84}{atmanim17@gmail.com}}

\address{ Laboratoire des Equations aux Dérivées Partielles\\ Algèbre et Géométrie Spectrale.
  }



\date{02/10/2021; 
\newline \indent $^{*}$Corresponding author}

\begin{abstract}
Controlled $\ast$-K-fusion frames are generalization of controlled fusion frames in frame theory. In this paper, we propose the notion of controlled $\ast$-k-fusions frames on Hilbert $C^{\ast}$-modules. We give some caraterizations and some of their properties are obtained. 
Then we study the erasures of submodules of a controlled $k$-fusion frame in Hilbert $C^{\ast}$-modules and we present some sufficient conditions under which a sequence remains a standart controlled k-fusion frame after deletion of some submodules. 
Finally, we introduce a perturbation for controlled $K$-fusion frames in Hilbert $C^{\ast}$-modules and it is shown that under some conditions controlled $K$-fusion frames are stable under this perturbation, and we generalize some of the results obtained for perturbations of controlled $K$-fusion frames.
\end{abstract}
  \maketitle

\section{Introduction}
In recent years the theory of frames has had an intensive development and many important contributions have been obtained. The first results on the class of frames started to appear in the early 1952s in the work of Duffin and Schaeffer to study nonharmonic Fourier series and widely studied by Daubechies, Grossman and Meyer in 1986. See (\cite{BPY06}, \cite{CA00}, \cite{CH03}, \cite{DGM86}, \cite{DS52}) for more information on frame theory and its applications. M. Frank and D. Larson generalized this notion to the situation of Hilbert $C^*$-modules. Hilbert $C^*$-modules are generalization of Hilbert spaces by allowing the inner product to take the values in a $C^*$-algebra rather than in the field of complex numbers. but the theory of Hilbert $C^*$-modules is different from the theory of Hilbert spaces, for exemple, no any closed submodule of a Hilbert $C^*$-modules is complemented.\\
Controlled frames have been recently introduced by Balazs et. al \cite{BA10} in Hilbert spaces to improve the numerical efficiency of interactive algorithms for inverting the frame operator in Hilbert spaces, however they are used earlier in\cite{Ab05}  for spherical wavelets.

When a huge amount of data has to be processed, it is often advantageous to subdivide a frame system into smaller subsystems and combine locally data vectors. This gives rise to the concept of fusion frames (or frames of subspaces)[\cite{CK08}] (see also [\cite{CH03}, Chapter 13]), which are a generalization of frames. They allow decompositions of the elements of H into weighted orthogonal projections onto closed subspaces.
Many concepts of classical frame theory have been generalized to the setting of fusion frames. $K$-fusion frames as a generalization of fusion frame, and they share many properties with the fusion frames, but there is different property, like the irreversibility of frame operator of $K$-fusion frames.

The paper is organized as follows. In Sec. $2$ we review the concept of Hilbert $C^*$-modules, fusion frames and controlled fusion frames  in Hilbert $C^*$-modules. In Sec. $3$, we introduce controlled $K$-fusion frames in Hilbert $C^*$-modules and characterize them,  also the analysis, synthesis  controlled $K$-$*$-frame operator will be reviewed,  and we show that they share many useful properties with their corresponding notions in Hilbert space. Next, we give conditions under which the removal of subspaces from controlled *-K-fusion frames leaves again a controlled *-K-fusion frames.
In Sec. $4$, we also introduce a new perturbation for  controlled $K$-fusion frames in Hilbert $C^*$-modules and it is shown that under some conditions controlled $K$-fusion frames are stable under this perturbation, and we generalize some of the results obtained for perturbations of controlled $K$-fusion frames.

\section{Preliminaries}
 In this section, we collect the basic notations and some preliminary results. We need to recall some notations and basic definitions.

 Throughout this paper, the symbol $I$ is reserved for a finite or countable index set, $\am$ is  a unital $C^*$-algebra with identity $1_{\am}$ and $\mathcal{H}\quad and\quad \mathcal{K}$ are Hilbert $C^*$-modules over $\mathcal{A}$. We reserve the notation $End^*_{\mathcal{A}}(\mathcal{H},\mathcal{K})$ for the set of all adjointable operators from $\mathcal{H}$ to $\mathcal{H}$. If $W$ is a closed submodule of $\mathcal{H}$ and $End^*_{\mathcal{A}}(\mathcal{H},\h)$ is abbreviated to $End^*_{\mathcal{A}}(\mathcal{H})$, we denote $\pi_{W}$ is the orthogonal projection of $\mathcal{H}$ into $W$. we say $W$ is orthogonally complemented if $\h= W\oplus W^{\perp}$ and in this case $\pi_{W}\in End^*_{\mathcal{A}}(\mathcal{H},W)$. And let $GL^+(\mathcal{H})$ be the set of all positive bounded linear invertible operators on $\h$ with bounded inverse.
 \begin{defi}\cite{LA95}
 Let $\mathcal{A}$ be a C$^*$-algebra and $\h$ be a left $\am$-module, such that the linear structures of $\am$ and $\h$ are compatible. A pre-Hilbert $\mathcal{A}$-module is a  left $\mathcal{A}$-module $\h$ equipped with a sesquilineare  form $\langle .,.\rangle:\quad \h\times \h \longrightarrow \mathcal{A}$ with the following properties:\begin{itemize}
 \item[i)] $\langle x,x\rangle_{\am} \geq 0$; $\langle x,x\rangle_{\am} =0$ if only if $x=0$ for all $x\in \h$,
 \item[ii)] $\langle ax, y\rangle_{\am}=a\langle x,y\rangle_{\am}$ for all $x,y\in \h$ and $ a\in \mathcal{A}$,
 \item[iii)] $\langle x, y\rangle_{\am}= \langle y, x\rangle_{\am}^*$ for all $x,y\in\h$.
 \end{itemize}
 \end{defi}
The map $\langle .,.\rangle$ is called an $\mathcal{A}$-valued inner product. For $x\in \h$, put $\Vert x \Vert_\h :=\Vert \langle x,x\rangle\Vert^{\frac{1}{2}}_{\am}$.
\\ A pre-Hilbert $\mathcal{A}$-module $\h$ is called a Hilbert $\am$-module or a Hilbert C$^*$-module over $\am$ if it is complete with respect its norme.\\
 We usually skip the subscript $\h$ when it does not lead to confusion of norms.\\
A C$^*$-algebra $\am$ itself can be recognized as a Hilbert $\am$-module with the $\am$-valued inner product $\langle a,b\rangle= ab^*$, for any $a,b\in \am$. \\
 The standard Hilbert $\am$-module $l_2(\mathcal{H},I)$ is defined by \begin{align*}
 l_2(\h,I):=\{ \{f_i\}_{i\in I}\subset \h, \sum_{i\in I}\langle f_i,f_i \rangle_{\am} \quad Converge \quad in \quad \Vert.\Vert\},
 \end{align*}
 whith the inner product $\langle \{f_i\}_{i\in I},\{g_i\}_{i\in I}\rangle=\sum_{i\in I}f_ig_i^*$, and we denote by the norm of this Hilbert C$^*$-module by $\Vert.\Vert_{l_2(\mathcal{H})}$.\\
For more detailes about Hilbert C$^*$ module, we refer the reader to \cite{LA95}.\\
For each element $a$ in $C^*$-algebra we have $\vert a\vert=(a^*a)^{\frac{1}{2}}$, and we define $\vert x\vert= (\langle x,x\rangle)^{\frac{1}{2}}.$ An element $a$ of a $C^*$-algebra $\am$ is positive if $a^*=a$ and its spectrum is a subset of positive real numbers.\\
 We call $\mathcal{Z}(\am)=\{ a\in \am\,:\, ab=ba\,\,\forall b\in \am\}$ the center of $\am$, if $a\in \mathcal{Z}(\am)$ then $a^*\in \mathcal{Z}(\am)$ element of $\mathcal{Z}(\am)$, then $a^{\frac{1}{2}}\in \mathcal{Z}(\am)$.
 
 \begin{defi} $($ See \cite{EL95} $)$ Let $\mathcal{B}(\mathcal{H},\mathcal{K})$ denote the set of all bounded $\mathcal{A}$-homomorphism operators from $\mathcal{H}$ to $\mathcal{K}$. we say that $T$ is adjointable if there exists a map  $T^*$ from $\mathcal{K}$ to $\mathcal{H}$ such that \begin{align*}
 \langle Tx,y \rangle_{\am}=\langle x, T^*y\rangle_{\am} \quad x\in \mathcal{H}, y\in \mathcal{K}.
 \end{align*}
 \end{defi}
 We denote by $End^*_{\mathcal{A}}(\mathcal{H},\mathcal{K})$ the set of all adjointable operators from $\mathcal{H}$ to $\mathcal{K}$.\\
 Every element of $End^*_{\mathcal{A}}(\mathcal{H},\mathcal{K})$ is a bounded $\mathcal{A}$-homomorphism but is important to realize that converse is false: a bounded $\mathcal{A}$-homomorphism need not to be adjointabe (see \cite{NW93}).
 \begin{defi}\cite{NW93}
 Let $\h$ and $\mathcal{K}$ be two Hilbert $C^*$-modules we say that a closed submodule $F$ of $\h$ is topologically complemented if there is a closed submodule $G$ of $\h$ with $\h = F + G,\quad F\cap
G = \{0\}\quad(briefly \quad H = F \oplus G)$. Furthermore, $F$ is said to be orthogonally complemented if $\h =F \oplus F^{\perp} $, where $F^{ \perp} = \{x \in \h \,\, such\, that \quad\langle x, y\rangle_{\am} = 0\,\,\, for\,\,\, any \,\,y \in F\}$.
 \end{defi}
By definition, if $F$ is orthogonally complemented then $F$ is topologically complemented, but the reverse is not true (see for a counterexample [\cite{NW93} page $7$]). However, an exception is that every nullspace of an element of $End^*_{\mathcal{A}}(\mathcal{H},\mathcal{K})$ with closed range is orthogonally complemented, which can be stated as follows:
\begin{lem} $($ See \cite{EL95} $)$
Let $\h$ and $\mathcal{K}$ be two Hilbert $\am$-modules, and $T \in End^*_{\mathcal{A}}(\mathcal{H},\mathcal{K})$. If $T$ has closed range, then $T^*$ also has closed range, and \begin{equation}
\h= Ker(T)\oplus Im(T^*),\quad \mathcal{K}=Ker(T^*)\oplus Im(T).
\end{equation}
\end{lem}
 \begin{prop} $($See \cite{NW93}$)$ If $T\in End^*_{\mathcal{A}}(\mathcal{H},\mathcal{K})$, then \begin{align}
 \langle T(h), T(h)\rangle_{\am} \leq \Vert T \Vert^2 \langle h,h\rangle_{\am}, h\in \h.
 \end{align}
 \end{prop}
 \begin{lem} $($See \cite{LA07}$)$ Let $\h$ and $\mathcal{K}$ be two Hilbert $\am$-modules,and $T \in End^*_{\mathcal{A}}(\mathcal{H},\mathcal{K})$. Then the following statemant are equivalent:
 \begin{itemize}
 \item[i)] $T$ is surjective.
 \item[ii)]$T^*$ is bounded below with respect to norm (i.e.) there is $m>0$ such that \begin{equation}\label{lem: operatornorme}
 \Vert T^*x\Vert\geq m\Vert x\Vert,\quad for\,\, all\,\, x\in\mathcal{K}.
 \end{equation}
 \item[iii)] $T^*$ is bounded below with respect to the inner product (i.e.) there is $m'>0$ such that \begin{equation}\label{lem: operatornorme}
 \langle T^*x,T^*x\rangle_{\am}\geq m'\langle x,x\rangle_{\am},\quad for\,\, all\,\, x\in\mathcal{K}.
 \end{equation}
 \end{itemize}
 \end{lem}
 \begin{lem}\cite{LA95}
 If $\varphi:\quad \am\longrightarrow \mathcal{B}$ is a $*$-homomorphism between $C^*$-algebras, then $\varphi$ is increasing, that is, if $a\leq b$, then $\varphi(a)\leq\varphi(b)$.
 \end{lem}
 \begin{lem} \cite{WP73}\label{lem: operatinverse} Let $\h$ and $\mathcal{K}$ be two Hilbert $\am$-modules,and $T \in End^*_{\mathcal{A}}(\mathcal{H},\mathcal{K})$. Then,\begin{itemize}
 \item[i)] If $T$ is injectif and $T$ has closed range, then the adjointable map $T^*T$ is invertible and\begin{equation*}
 \Vert (T^*T)^{-1}\Vert \mathcal{I}_{\mathcal{K}}\leq T^*T\leq \Vert T\Vert^2 \mathcal{I}_{\mathcal{K}}.
 \end{equation*}
 \item[ii)] If $T$ is surjective, then adjointable map $TT^*$ is invertible and \begin{equation*}
 \Vert (TT^*)^{-1}\Vert^{-1} \mathcal{I}_{\mathcal{K}}\leq TT^*\leq \Vert T\Vert^2 \mathcal{I}_{\mathcal{K}}.
 \end{equation*}
 \end{itemize}
 
 \end{lem}
 \begin{lem}$( See$ \cite{FY09}$)$ \label{lem: ortocomplope} Let $\h,\mathcal{K}$ and $\mathcal{G}$ be Hilbert $\am $-modules over a $C^*$-algebra $\am$. let $T\in End^*_{\mathcal{A}}(\mathcal{H},\mathcal{K})$ and $T' \in End^*_{\mathcal{A}}(\mathcal{G},\mathcal{K})$ with $\overline{ Im(T^*)}$ be orthogonally complemented. Then the following statements are equivalent:
 \begin{itemize}
 \item[i)] $ T'T'^*\leq \lambda TT^* $ for some $\lambda>0$;
 \item[ii)] There exists $\mu>0$ such that $\Vert T'^*z\Vert\leq \Vert \mu \Vert T^*z\Vert$, for all $z\in \mathcal{K}$.
 \item[iii)] $Im(T')\subset Im(T)$.
 \end{itemize}
 \end{lem}
 \begin{pro}$( See$ \cite{XS08}$)$
 Let $M$ be an orthogonally complemented submodules of a Hilbert $A$-module $\h$ and let $T\in End^*_{\am}(\h)$ be an invertible operator  such that  $T^*TM\subset M$. Then we have: \begin{align}
 T(M^{\perp})=(TM)^{\perp},\quad \pi_{TM}=T\pi_{M} T^{-1}.
\end{align}  
 \end{pro}

 It has been proven that an operator between two Hilbert $C^*$-module admits a Moore-Penrose inverse if and only if it has closed range $( See$ \cite{KS10} $)$
 \begin{defi} Let $T\in End^*_{\mathcal{A}}(\mathcal{H},\mathcal{K})$ be a closed range, then  there exists an adjointable operator $T^\dagger\in End^*_{\mathcal{A}}(\mathcal{K},\mathcal{H})$  called the Moore-Penrose inverse of $T$ if \begin{align}\label{lem: pseudo}
 TT^{\dagger}T=T,\quad T^{\dagger}TT^{\dagger}=T^{\dagger},\quad (TT^{\dagger})^*=TT^{\dagger},\, and\, (T^{\dagger}T)^*=T^{\dagger}T.
 \end{align}

 \end{defi}
 \subsection*{Fusion frame} 
 \begin{defi}\cite{KH08}
 Let $\h$ be an Hilbert $C^*$-module over $\mathcal{A}$ and $\{w_i\}_{i\in I}$ be a family of weights in $\h$, i.e., each $w_i$ is a positive invertible element from the center of the $C^*$-algebra $\mathcal{A}$. A family of closed submodules $\mathcal{W}=\{W_i,w_i\}_{i\in I}$ is a fusion frame for $\h $ if every $W_i$ is orthogonally complemented and there exist real constants $ 0<A\leq B<\infty $ such that  \begin{equation}\label{def: kf}
A \langle f,f\rangle_{\am}\leq\sum_{i\in I} w_i^2 \langle\pi_{W_i}f,\pi_{W_i} f \rangle_{\am} \leq B \langle f, f\rangle_{\am}, f \in\mathcal{H}.
\end{equation}
 \end{defi}
 We call $A$ and $B$ the lower and upper bounds of the fusion frame. the fusion frame $\{W_i,w_i\}_{i\in I}$ is said to be $B$-tight if $A=B$, and said to be Perrseval if $A=B=1$. If the right -handed inequality of holds, then we say that $\{W_i,w_i\}_{i\in I}$ is a Bessel fusion sequence for $\mathcal{H}$ with Bessel fusion bound $B$.\\
 
 Let $\{W_i,w_i\}_{i\in I}$ be a fusion frame for $\h$, then the fusion operator $ S:\quad \h\longrightarrow \h$ defined by \begin{equation}\label{def: fos}
Sx=\sum_{i\in I} w_i^2 \pi_{W_i}(x),\quad x\in \h,
 \end{equation}
 is a positive, self-adjoint and invertible operator. These properties imply that $\{W_i,w_i\}_{i\in I}$ admits the reconstruction formular \begin{equation*}
 x=\sum_{i\in I} w_i^2 S^{-1}\pi_{W_i}(x)=\sum_{i\in I} w_i^2 \pi_{W_i}(S^{-1}x),\quad x\in \h.
 \end{equation*}
 $($Throughout the paper, the sums like those in the middles of (\ref{def: fos}) and (\ref{def: kf}) are assumed to be convergent
in the norm sense.)
\subsection{$*$-$K$-fusion Frame for Hilbert $C^*$-modules}
\begin{defi}
Let $K\in End^*_{\mathcal{A}}(\mathcal{H})$, and $\{W_i\}_{i\in I}$ be a family of orthogonally complemented submodules of Hilbert $\am$-module $\mathcal{H}$. Also, let $\{w_i\}_{i\in I}$ be a family of weights i.e.,, each $w_j>0$ is a positive, invertible element from the center of $\am$. Then  $\mathcal{W}=\{W_i,w_i\}_{i\in I}$ is called a $*$-$K$-fusion frame if there exist two strictly nonzero elements $A$ and $B$ in $\am$ such that 
 \begin{equation} \label{def: *ff}
A \langle K^*f,  K^*f\rangle_{\am} A^*\leq\sum_{i\in I} w_i^2\langle\pi_{W_i}f,\pi_{W_i} f\rangle_{\am}\leq B \langle f,f\rangle_{\am} B^* , 
\end{equation}
For all $f \in\mathcal{H}$.\\
 \end{defi}

 Let $\{W_i,w_i\}_{i\in I}$ be a $*$-$K$-fusion frame for $\h$, with bounds $A$ and $B$. Then, the corresponding fusion frame synthesis operator $T_{W, v}: \h \rightarrow \ell_{2}(I, \h)$ defined by $T_{W, v}(f)=\left(v_{i} \pi_{W_{i}}(f)\right)_{i \in I}$ for $f \in \h$, and its adjoint operator $T_{W, v}^{*}: \ell_{2}(I, \h) \rightarrow \h$ is  defined by $T_{M, v}^{*}(y)=\sum_{i \in I} v_{i} \pi_{W_{i}}\left(y_{i}\right)$ for all $g=\left(g_{i}\right)_{i \in I}$ in $\ell_{2}(I, X).$ Then the $*$-fusion frame operator $ S:\quad \h\longrightarrow \h$ defined by \begin{equation}\label{def: etoilefos}
Sf=\sum_{i\in I} w_i^2 \pi_{W_i}(f),\quad f\in \h,
 \end{equation}
  It is clear that $S$ is well defined, positive, bounded and self-adjoint  operator.
\subsection{Controlled $*$-fusion Frame for Hilbert $C^*$-modules}
\begin{defi}
Let $\{W_i\}_{i\in I}$ be a family of orthogonally complemented submodules of Hilbert $\am$-module $\mathcal{H}$. Also, let $\{w_i\}_{i\in I}$ be a family of weights i.e., each $w_j>0$ is a positive, invertible element from the center of $\am$, and let $C, C'\in GL^+(\mathcal{H})$. Then  $\mathcal{W}=\{W_i,w_i\}_{i\in I}$ is called a $*$-fusion frame controlled by $C$ and $C'$ or $(C,C')$-controlled $*$-fusion frame if there exist two strictly nonzero elements $A$ and $B$ in $\am$ such that 
 \begin{equation} \label{def: c*ff}
A \langle f,  f\rangle_{\am} A^*\leq\sum_{i\in I} w_i^2\langle\pi_{W_i}C f,\pi_{W_i}C' f\rangle_{\am}\leq B \langle f,f\rangle_{\am} B^* , 
\end{equation}
For all $f \in\mathcal{H}$.
\end{defi}
Let  $\mathcal{W}$ is a $(C,C')$-controlled $*$-fusion frame and  $C'^*\pi_{W_i}C$ is a positive operator for each $i\in I$, then $C'^*\pi_{W_i}C=C^*\pi_{W_i} C'$.\\

  \begin{eqnarray*}
  \sum_{i\in I} w_i^2 \langle \pi_{W_i}C f,\pi_{W_i}C' f\rangle _{\am}&=& \sum_{i\in I} w_i^2  \langle C'^*\pi_{W_i}C f, f\rangle_{\am}\\
  &=& \sum_{i\in I} w_i^2 \langle (C'^*\pi_{W_i}C)^{\frac{1}{2}}f, (C'^*\pi_{W_i}C)^{\frac{1}{2}}f \rangle_{\am}.\\
  \end{eqnarray*}
 We define the controlled analysis operator by 
 \begin{eqnarray*}
  T:\quad \mathcal{H} &\longrightarrow & \mathcal{K}\\ f &\longmapsto &T(f):=(w_i (C'^*\pi_{W_i}C)^{\frac{1}{2}}f)_{i\in I},
\end{eqnarray*}
Where\begin{align*}
 \mathcal{K}=\{(w_i (C'^*\pi_{W_i}C)^{\frac{1}{2}}f)_{i\in I}  \vert f\in \mathcal{H}\}\subseteq (\oplus_{i\in I}\mathcal{H})_{l^2}.
 \end{align*}
 $\mathcal{K}$ is closed and $ T$ is well defined. Morever $T$ is a bounded linear operator. Its adjoint operator is given by   \begin{eqnarray*} \label{def: adjoint}
  T^*:\quad \mathcal{K} &\longrightarrow & \mathcal{H}\\(w_i (C'^*\pi_{W_i}C)^{\frac{1}{2}}f)_{i\in I}) &\longmapsto &T^*((w_i (C'^*\pi_{W_i}C)^{\frac{1}{2}}f)_{i\in I}):=\sum_{i\in I} w_i^2C'^*\pi_{W_i}C f,
\end{eqnarray*}
and it is called the controlled synthesis operator.\\
Therefore, we define the controlled  $*$-fusion frame operator $S$ on $\mathcal{H}$ by\begin{align}\label{def:opers}
S_{CC'}=T^*T(f)=\sum_{i\in I} w_i^2 C'^*\pi_{W_i}C f, \quad  f\in\mathcal{H}.
\end{align}

As for frames, $\{W_i,w_i\}_{i\in I}$ is a fusion sequence for $\h$ if and only if the synthesis operator is a well defined, bounded linear operator. 
\begin{thm}\label{th: ckfsoperator} $\mathcal{W}$ is a $(C,C')$-controlled  $*$-fusion Bessel sequence for $\h$ with bound $B$ if and only if $T^*$ is well-defined bounded operator and $\Vert T^*\Vert \leq \Vert B\Vert$.
\end{thm}
\begin{proof} Assume that  $\{W_i,w_i\}_{i\in I}$ is a $(C,C')$-controlled  $*$-fusion Bessel sequence for $\h$ with bound $B$. Let $(w_i (C'^*\pi_{W_i}C)^{\frac{1}{2}}f)_{i\in I}\in \mathcal{K}$. First, we want to show that $T^*$ is well-defined, ie., that the serie $\sum_{i\in I} w_i^2C'^*\pi_{W_i}C f$ is convergent. Consider $n,m \in I$, for $n>m$. Then, we have 
\begin{eqnarray*}
\Vert \sum_{i=m+1}^{i=n}w_i^2(C'^* \pi_{W_i}C) f\Vert ^2 &=&\sup_{g\in \h, \Vert g\Vert=1}\Vert \langle \sum_{i=m+1}^{i=n} w_i^2(C'^* \pi_{W_i}Cf,g \rangle_{\am}\Vert^2\\
&=& \sup_{g\in \h, \Vert g\Vert=1}\Vert \sum_{i=m+1}^{i=n}\langle w_i(C'^* \pi_{W_i}C)^{\frac{1}{2}}f,w_i(C'^* \pi_{W_i}C)^{\frac{1}{2}}g \rangle_{\am}\Vert^2\\
&\leq&  \sup_{f\in \h, \Vert f\Vert=1}\Vert  \sum_{i=m+1}^{i=n}\langle  w_i(C'^* \pi_{W_i}C)^{\frac{1}{2}}f, w_i(C'^* \pi_{W_i}C)^{\frac{1}{2}}f\rangle \Vert^2\\&& \Vert  \sum_{i=m+1}^{i=n}\langle w_i (C'^*\pi_{W_i}C)^{\frac{1}{2}}g,w_i (C'^*\pi_{W_i}C)^{\frac{1}{2}}g\rangle_{\am} \Vert^2\\
&=& \sup_{f\in \h, \Vert f\Vert=1}\Vert  \sum_{i=m+1}^{i=n}  w_i^2 \langle \pi_{W_i}Cf, \pi_{W_i}C'f\rangle_{\am} \Vert^2 \\&&  \Vert  \sum_{i=m+1}^{i=n}  w_i^2\langle  w_i(C'^* \pi_{W_i}C)^{\frac{1}{2}}g,w_i(C'^* \pi_{W_i}C)^{\frac{1}{2}}g\rangle \Vert^2\\
&\leq &\Vert B\Vert^2 \Vert \{w_i (C'^*\pi_{W_i}C)^{\frac{1}{2}} f\}_{i\in I} \Vert_{l^2(\h)}.
\end{eqnarray*}
Since the right hand side tends to zero as $n\longrightarrow +\infty$, then the sequence satisfies the cauchy criterion and also converges. So the operator $ T^*$ is well-defined.

For the opposite imlication, suppose that $ T^*$ is well-defined and that $\Vert T^* \Vert \leq \Vert B\Vert$. Let us prove that the system $\{W_i,w\}_{i\in I}$ is a  $(C,C')$-controlled   $*$-fusion Bessel sequence in $\h$ with a bound $B$. For arbitrary  $f\in\h$ and $(w_i (C'^*\pi_{W_i}C)^{\frac{1}{2}}f)_{i\in I}\in \mathcal{K}$, we have 
\begin{eqnarray*}
\sum_{i\in I} \langle \pi_{W_i}C f,\pi_{W_i}C' f \rangle_{\am} &=& \sum_{i\in I}\langle w_i^2 (C'^*\pi_{W_i}C)f ,f\rangle_{\am}\\
&=& \langle \{w_i ( C'^*\pi_{W_i}C)^{\frac{1}{2}}\}_{i\in I}f,\{w_i (C'^*\pi_{W_i}C)^{\frac{1}{2}}\}_{i\in I}f \rangle_{l^2(\h)}\\
&\leq &  \Vert \{w_i (C'^*\pi_{W_i}C)^{\frac{1}{2}}\}_{i\in I}\Vert^2 \langle f,f\rangle_{l^2(\h)}\\
&\leq &  \Vert B\Vert \langle f,f\rangle_{\am}.
\end{eqnarray*}
\end{proof}

\begin{thm}\textit{(Reconstruction formula)}\label{th:reconstructionformula1} Let $K\in End^*_{\mathcal{A}}(\mathcal{H})$ be a surjective operator, if $\{W_i,w_i\}_{i\in I}$ is a $(C,C')$-controlled $*$-fusion frame for $\h$. Then, the controlled $*$-fusion  frame operator $S_{CC'}$ is positive, invertible and adjointable. Morever, we have the reconstruction formula $f=\sum_{i\in I} w_i^2 C'^{*}\pi_{W_i}C S_{CC'}^{-1}f,\quad f\in \h,$
\end{thm}
\begin{proof} By using the definition of $(C,C')$-controlled $*$-fusion frame and Lemma (\ref{th: ckfsoperator}). Then, it is easy to check that the controlled $*$-fusion  frame operator  is well defined, positive and adjointable. Morever, we have
\begin{align}\label{proof: norm}
\Vert Tf\Vert^2=\left\|\left\{ v_i C'^* \pi_{V_{i}} C f\right\}_{i \in I}\right\|_{l^2(\h)}= \Vert sum_{i\in I} w_i^2\langle\pi_{W_i}C f,\pi_{W_i}C' f\rangle_{\am}\Vert \leq \Vert B\Vert^2 \Vert\langle f,f\rangle_{\am} \Vert^2.
\end{align} 
Now, we want to show that the range of $T$ is closed. Let $\{Tf_n\}_{n\in\mathbb{N}}$ be a sequence in the range of $T$ such that $\lim _{n \rightarrow \infty} T f_{n}=g.$ For $n, m \in \mathbb{N}$, we have
$$
\left\|A_{1}\left\langle f_{n}-f_{m}, f_{n}-f_{m}\right\rangle A_{1}^{*}\right\| \leq\left\|\left\langle T\left(f_{n}-f_{m}\right), T\left(f_{n}-f_{m}\right)\right\rangle\right\|=\left\|T\left(f_{n}-f_{m}\right)\right\|^{2}.
$$
Since, $\left\{T f_{n}\right\}_{n \in \mathbb{N}}$ is a Cauchy sequence in $\mathcal{K}$,
$$
\left\|A_{1}\left\langle f_{n}-f_{m}, f_{n}-f_{m}\right\rangle A_{1}^{*}\right\| \rightarrow 0, \quad \text { as } n, m \rightarrow \infty,
$$
for $n, m \in \mathbb{N}$, we have
$$
\begin{aligned}
\left\|\left\langle f_{n}-f_{m}, f_{n}-f_{m}\right\rangle\right\| &=\left\|A_{1}^{-1} A_{1}\left\langle f_{n}-f_{m}, f_{n}-f_{m}\right\rangle A_{1}^{*}\left(A_{1}^{*}\right)^{-1}\right\| \\
& \leq\left\|A_{1}^{-1}\right\|^{2}\left\|A_{1}\left\langle f_{n}-f_{m}, f_{n}-f_{m}\right\rangle A_{1}^{*}\right\|.
\end{aligned}
$$
This means that $\left\{f_{n}\right\}_{n \in \mathbb{N}}$ is a Cauchy sequence, hence there exists $f \in \h$ such that $f_{n} \rightarrow f$ by letting $n \rightarrow \infty .$ Now by (\ref{def: adjoint}) we have $\left\|T\left(f_{n}-f\right)\right\|^{2} \leq\|B\|^{2}\left\|\left\langle f_{n}-f, f_{n}-f\right\rangle\right\| .$ Thus $\left\|T f_{n}-T f\right\| \rightarrow 0$ as $n \rightarrow \infty$ implies that $Tf=g .$ It concludes that the range of $T$ is closed. Hence one easily shows that $T$ is injective. Using the injectivity of $T$, the operator $T^{*}$ has a closed range and $\mathcal{H}=\operatorname{range}\left(T^{*}\right)$, since $S_{CC'}=T^{*} T$, then  the invertibility of $S_{CC'}$ can be derived by Lemma (\ref{lem: operatinverse}).  

\end{proof}

\section{Controlled $*$-$K$-Fusion Frame for Hilbert $C^*$-modules}
Controlled $*$-$K$-Fusion Frames are $C^*$-algebra version of Controlled $K$-Fusion Frames. In fact, we need strictly positive elements of  $C^*$-algebra $\am$ instead of positive real numbers.
\begin{defi}
Let $K\in End^*_{\mathcal{A}}(\mathcal{H})$, and $\{W_i\}_{i\in I}$ be a family of orthogonally complemented submodules of Hilbert $\am$-module $\mathcal{H}$. Also, let $\{w_i\}_{i\in I}$ be a family of weights in $\am$, i.e., each $w_j>0$ is a positive, invertible element from the center of $\am$ and let $C, C'\in GL^+(\mathcal{H})$. Then  $\mathcal{W}=\{W_i,w_i\}_{i\in I}$ is called a $*$-$K$- fusion frame controlled by $C$ and $C'$ or $(C,C')$-controlled $*$-$K$-fusion frame if there exist two strictly nonzero elements $A$ and $B$ in $\am$ such that 
 \begin{equation} \label{def: c*kff}
A \langle K^* f, K^* f\rangle_{\am} A^*\leq\sum_{i\in I} w_i^2\langle\pi_{W_i}C f,\pi_{W_i}C' f\rangle_{\am} \leq B \langle f,f\rangle_{\am} B^*, 
\end{equation}
For all $f \in\mathcal{H}$.\\
Where $K^*$ is the adjoint operator of $K$. 
\end{defi}
$A$ and $B$ are called lower and upper bounds of a $(C,C')$-controlled $*$-$K$-fusion frame respectively. Since $\am$ is not a partial ordered set lower and upper bounds may not exist
\begin{enumerate}
\item We call $\mathcal{W}$ a $(C,C')$-controlled Parsval  $*$-$K$-fusion frame if \\$A_{CC'}=B_{CC'}=1$.
\item If only the second inequality (\ref{def: c*kff}) is required, we call $\mathcal{W}$ a $(C,C')$-controlled  Bessel $*$-$K$-fusion sequence with Bessel bound $B$.
\end{enumerate}

We mentioned that the set of all of controlled $K$-fusion frames for $End^*_{\am}(\h)$ can be considered as a subset of controlled $*$-$K$-fusion frame. To illustrate this, let $\{W_i,w_i\}_{i\in I}$ be a controlled $K$-fusion frames for Hilbert $\am$-module $\h$ with  real bounds $A$ and $B$. Note that for $f\in \h$,\begin{equation*}
\sqrt{A}1_{\am} \langle K^* f, K^* f\rangle (\sqrt{A}1_{\am})^*\leq\sum_{i\in I} w_i<\pi_{W_i}C f,\pi_{W_i}C' f>  w_i^*\leq \sqrt{A}1_{\am} \langle f,f\rangle(\sqrt{A}1_{\am})^*. 
\end{equation*}
Therefore, every controlled $K$-fusion frames $\h$ with real bounds $A$ and $B$ is a controlled $*$-$K$-fusion frame for $\h$ with $\am$-valued controlled $*$-$K$-fusion frame bounds $(\sqrt{A} 1_{\am})$ and $(\sqrt{B} 1_{\am})$.

The following example shows that in general, frames may be controlled $*$-$K$-fusion frame  without being a controlled $*$-fusion frame in Hilbert $\am$-module. 
\begin{exa} Let \begin{align*}
\mathcal{H}= l^2(\mathbb{N})=\{ \{a_n\}_{n\in \mathbb{N}}\subset \mathbb{C}: \quad \sum_{n\in \mathbb{N}}\vert a_n\vert^2\leq \infty\}, \quad and
\end{align*} 
\begin{align*}
\am=l^\infty(\mathbb{N})=\{ \{a_n\}_{n\in \mathbb{N}}\subset \mathbb{C}: \quad \sup_{n\in \mathbb{N}}\vert a_n\vert^2\leq \infty\}.
\end{align*}
It is known that $\am$ is a commutative, unitale $C^*$-algebra and $\h$ is a Hilbert $\am$-module.
\begin{align*}
W_i=\{ \{a_n\}_{n\in \mathbb{N}}\subset \mathbb{C}: \quad a_j=0\quad \forall j\neq 2i\}.
\end{align*}
\begin{align*}
 \{W_n\}_{n\in \mathbb{N}}\subset l^2(\mathbb{N})\cap \mathcal{\am}, \quad \sum_{n\in \mathbb{N}}\vert a_n\vert^2\leq \infty
\end{align*}
Let $C$ and $C'$ be two operators, respectively defined as follows:
\begin{eqnarray*}
C:\quad \h &\longrightarrow& \h\\
\{a_n\}_{n\in \mathbb{N}}&\longmapsto&\alpha \{a_n\}_{n\in \mathbb{N}},
\end{eqnarray*}
\begin{eqnarray*}
C':\quad \h &\longrightarrow& \h\\
\{a_n\}_{n\in \mathbb{N}}&\longmapsto&\beta\{a_n\}_{n\in \mathbb{N}},
\end{eqnarray*}
where $\alpha$ and $\beta $ are two real numbers strictly positive.\\
It is clear that  $C,C'\in GL^+(\h)$. Indeed, for each $\{a_n\}_{n\in \mathbb{N}}\in\h$, one has
\begin{eqnarray*}
C^{-1}(\{a_n\}_{n\in \mathbb{N}})&=&\alpha^{-1}\{a_n\}_{n\in \mathbb{N}},\\
C'^{-1}(\{a_n\}_{n\in \mathbb{N}})&=&\beta^{-1}\{a_n\}_{n\in \mathbb{N}}.\\
\end{eqnarray*}
 $ \pi_{W_j}(\{a_n\}_{n\in \mathbb{N}})= \{b_n\}_{n\in \mathbb{N}}^j $, avec $b_n=0,\quad\forall n\neq 2j$. 
  \begin{align*}
\pi_{W_j}(\{a_n\}_{n\in \mathbb{N}})=( 0,0,...,0,a_{2j},0,..,).  
  \end{align*}
\begin{eqnarray*}
\langle\pi_{W_j}C(\{a_n\}_{n\in \mathbb{N}}),\pi_{W_i}C'(\{a_n\}_{n\in \mathbb{N}})\rangle_{\am}&=&\langle \alpha\{b_n\}_{n\in \mathbb{N}}^j ,\beta\{b_n\}_{n\in \mathbb{N}}^j \rangle_{\am}\\ &=&\alpha\beta( 0,0,...,0,\vert a_{2j}\vert^2,0,..,0).
\end{eqnarray*}
\begin{align*}
\sum_{i\in I} w_i^2\langle\pi_{W_j}C(\{a_n\}_{n\in \mathbb{N}}),\pi_{W_i}C'(\{a_n\}_{n\in \mathbb{N}})\rangle_{\am}=\sum_{i\in I} w_i^2\alpha\beta( 0,0,...,0,\vert a_{2j}\vert^2,0,..,).
\end{align*}
$\{W_i,w_i\}_{i\in I}$ is not a controlled $*$-fusion frame. Now if we considere  the operator \begin{eqnarray*} K :\quad \h &\longrightarrow& \h\\
\{a_n\}_{n\in \mathbb{N}}&\longmapsto&\beta\{b_n\}_{n\in \mathbb{N}},
\end{eqnarray*}
where $$b_{2n} =\left\lbrace\begin{array}{lll}
w_{2n}\frac{a_{2n}}{2} ;\\
0 & sinon. &
\end{array}\right. $$
$\{W_i,w_i\}_{i\in I}$ becomes a controlled $*$-$K$-fusion frame, indeed, for each  $\{a_n\}_{n\in \mathbb{N}}\in\h$, we have
\begin{eqnarray*}
\langle K^* (\{a_n\}_{n\in \mathbb{N}}),K^*(\{a_n\}_{n\in \mathbb{N}})\rangle _{\am}&=& \frac{1}{4}\langle \alpha(w_{2n}\{a_{2n}\}_{n\in \mathbb{N}}),\beta(w_{2n}\{a_{2n}\}_{n\in \mathbb{N}})\rangle_{\am}\\ &=&\frac{\alpha\beta}{4}(w_{2n}^2\vert a_n\vert^2)_{n\in \mathbb{N}}\\&\leq& \frac{1}{4} \sum_{i\in I} w_i^2\langle\pi_{W_i}C f,\pi_{W_i}C' f\rangle_{\am}.
\end{eqnarray*}
It follows that 
\begin{align*}
2.1_{\am}\langle K^* (\{a_n\}_{n\in \mathbb{N}}),K^*(\{a_n\}_{n\in \mathbb{N}})\rangle _{\am} 2.1_{\am} \leq \sum_{i\in I} w_i^2\langle\pi_{W_i}C f,\pi_{W_i}C' f\rangle_{\am}.
\end{align*}
\end{exa}
\begin{rem}\label{rem: relation}
If $K\in End^*_{\am}(\h)$ is a surjective operator, then every controlled $*$-$K$-fusion frame is a controlled $*$-fusion frame for $\h$.
\end{rem} 
\subsection{Some characterizations of controlled $*$-$K$-fusion frame in Hilbert C*-module}
In this section, we will character the equivalencies of controlled $*$-$K$-fusion frame in Hilbert $C^*$-module from several aspects. As for Theorem (\ref{th: characterizationnorm}), we show that the controlled $*$-$K$-fusion frame is equivalent to which the middle of (\ref{def: c*kff}) is norm bounded (see (\ref{def: cNkff})).  the analysis, synthesis controlled $*$-$K$-fusion frame operator are defined the same as in controlled $*$-fusion frame case.

\begin{thm}\label{th: characterizationnorm} Let $K\in End^*_{\am}(\h)$ be a closed range, Let $\{W_i\}_{i\in I}$ be a family of orthogonally complemented submodules of Hilbert $\am$-module $\mathcal{H}$. Also, let $\{w_i\}_{i\in I}$ be a family of weights, and let $C, C'\in GL^+(\mathcal{H})$ such that  $\sum_{i\in I} w_i^2\langle\pi_{W_i}C f,\pi_{W_i}C' f\rangle_{\am}$ converge in norm $\am$. Then $\mathcal{W}=\{W_i,w_i\}_{i\in I}$ is a $(C,C')$-controlled $*$-$K$-fusion frame if and only if
\begin{equation} \label{def: cNkff}
\Vert A^{-1}\Vert^{-2}\Vert \langle K^* f, K^* f\rangle_{\am}\Vert^2\leq\Vert\sum_{i\in I} w_i^2\langle\pi_{W_i}C f,\pi_{W_i}C' f\rangle_{\am}\Vert \leq \Vert B\Vert^2 \Vert\langle f,f\rangle_{\am} \Vert^2, 
\end{equation}
for all $f \in\mathcal{H}$ and strictly nonzero elements $A,B$ in $\am$.
\end{thm}
\begin{proof}
Assume that $\{W_i,w_i\}_{i\in I}$ is a controlled $*$-$K$-fusion frame, then for $f\in\h$ we have 
\begin{eqnarray*}
\langle K^*f,K^*f \rangle_{\am} &\leq & A^{-1}\langle S_{(C,C')}f, f\rangle_{\am}(A^*)^{-1},  \\
\langle S_{(C,C')}f,f \rangle_{\am} &\leq & B\langle f,f\rangle_{\am} B^*.
\end{eqnarray*}
Hence,
\begin{equation*} \label{def: cNkff}
\Vert A^{-1}\Vert^{-2}\Vert \langle K^* f, K^* f\rangle_{\am}\Vert\leq\Vert\sum_{i\in I} w_i^2\langle\pi_{W_i}C f,\pi_{W_i}C' f\rangle_{\am}\Vert \leq \Vert B\Vert^2 \Vert\langle f,f\rangle_{\am} \Vert. 
\end{equation*}
Conversely, suppose that (\ref{def: cNkff}) holds. Since $S_{(C,C')}$ is a positive, self adjoint, and invertible  controlled $*$-$K$-fusion frame operator, then  for each $f\in\h$ we have \begin{eqnarray*}
\langle (S_{(C,C')})^{\frac{1}{2}}f,  (S_{(C,C')})^{\frac{1}{2}}f\rangle_{\am} &=& \langle S_{(C,C')}f,f\rangle_{\am}\\
&=&\sum_{i\in I} w_i^2\langle\pi_{W_i}C f,\pi_{W_i}C' f\rangle_{\am}.
\end{eqnarray*}
Combining this with (\ref{def: cNkff}), we obtain \begin{align}\label{def: cSkff}
\Vert A^{-1}\Vert^{-2}\Vert K^*f\Vert^2 \leq \Vert (S_{(C,C')})^{\frac{1}{2}}f\Vert^2 \leq \Vert B\Vert^2 \Vert f\Vert^2.
\end{align}
By using (\ref{def: cSkff}) and lemma \cite{LA07}, we conclude that $\{W_i,w_i\}_{i\in I}$ is a $ (C,C')$-controlled $*$-$K$-fusion frame.
\end{proof}
\begin{prop} Assume that $C$ and $K^*$ commute with each other.  
The family $\{W_i,w_i\}_{i\in I}$ is a controlled $*$-$K$-fusion frame for $\h$ if and only if $\{W_i,w_i\}_{i\in I}$ is a $ (C,C)$-controlled $*$-$K$-fusion frame for $\h$.
\end{prop}
\begin{proof} Suppose that $\{W_i,w_i\}_{i\in I}$ is a $(C,C)$-controlled $*$-$K$-fusion frame with bounds $A$ and $B$. by definition, we have  \begin{align*}
\sum_{i\in I} w_i^2\langle\pi_{W_i}C f,\pi_{W_i}C f\rangle_{\am} \leq B\langle f,f\rangle_{\am}A^*,
\end{align*}
On the one hand, for all $f\in\h$, we have 
\begin{eqnarray*}
A\langle K^*f, K^*f\rangle_{\am} A^*&=& A\langle K^*CC^{-1}f, K^* CC^{-1}f \rangle_{\am} A^*\\
&\leq & A\Vert C\Vert^2 \langle K^*C^{-1}f, K^*C^{-1}f\rangle_{\am} A^*\\
&\leq & \Vert C\Vert^2 \sum_{i\in I} w_i^2\langle\pi_{W_i}C C^{-1}f,\pi_{W_i}CC^{-1} f\rangle_{\am} \\
& = &  \Vert C\Vert^2 \sum_{i\in I} w_i^2\langle\pi_{W_i}f,\pi_{W_i} f\rangle_{\am}.
\end{eqnarray*}
Then,
\begin{align}\label{prp: al1 }
(\Vert C\Vert^{-1}A)\langle K^*f, K^*f\rangle_{\am} (\Vert C\Vert^{-1}A)^*\leq  \sum_{i\in I} w_i^2\langle\pi_{W_i}f,\pi_{W_i} f\rangle_{\am}.
\end{align}
On the other hand, for any $f\in \h$, we have 
\begin{eqnarray*}
\sum_{i\in I} w_i^2\langle\pi_{W_i}f,\pi_{W_i} f\rangle_{\am} &=& \sum_{i\in I} w_i^2\langle CC^{-1}\pi_{W_i}f,CC^{-1}\pi_{W_i} f\rangle_{\am}\\ &\leq & B \langle C^{-1}f,C^{-1}f\rangle_{\am}B^*\\
&\leq& (\Vert C^{-1}\Vert B) \langle f,f\rangle_{\am} (\Vert C^{-1}\Vert B)^*.
\end{eqnarray*} 
For the converse, for all $f\in \h$ we have
\begin{align*}
A'\langle CC^{-1}f,CC^{-1}f\rangle_{\am}A'^* \leq \Vert C^{-1}\Vert^2 \sum_{i\in I} w_i^2\langle\pi_{W_i}Cf,\pi_{W_i}C f\rangle_{\am}.
\end{align*}
\end{proof}

 Controlled K-fusion frame operator of $(C,C')$-controlled $K$-fusion frame is not invertible in general, but we can show that it is invertible on the subspace $Im(K) \subset \mathcal{H}$. In fact, since $Im(K)$ is closed \begin{align*}
 KK^\dagger\mid_{Im(K)}= id_{Im(K)},
 \end{align*}
so we have \begin{align*}
 id^*_{Im(K)}=(K^\dagger\mid_{Im(K)})^*K^*.
\end{align*}
Hence for any $f\in Im(K)$
\begin{eqnarray*}
\Vert f\Vert =\Vert (K^\dagger\mid_{Im(K)})^*K^* f\Vert\leq \Vert K^\dagger\Vert \Vert K^* f\Vert,
\end{eqnarray*}
that is, $\Vert K^* f\Vert^2\geq \Vert K^\dagger\Vert^{-2}\Vert \Vert f \Vert^2$. Combined with (\ref{def: cNkff}) we have \begin{align*}
\langle S_{CC'}f,f\rangle_{\am} \geq A_{CC'}\Vert K^* f\Vert^2\geq A_{CC'}\Vert K^\dagger\Vert^{2}\Vert f\Vert^2,\quad \forall f\in Im(K). 
\end{align*}
So from the definition of $(C,C')$-controlled $K$-fusion frame, which implies that $S:\quad Im(K)\longrightarrow S(Im(K))$ is an $*$-isomorphism.\\
Thus, by using Theorem (\ref{th:reconstructionformula1}) and Remark (\ref{rem: relation}), we immediately obtain the following result. 

\textbf{\textit{Remark}:} If $\{W_i,w_i\}_{i\in I}$ is a $(C,C')$-controlled $*$-$K$-fusion frame for $\h$ and $K\in End^*_{\mathcal{A}}(\mathcal{H})$ be a surjective operator. Then the frame operator $S_{CC'}$ is positive, invertible and adjointable. Morever, we have the reconstruction formula $f=\sum_{i\in I} w_i^2 C'^{*}\pi_{W_i}C S_{CC'}^{-1}f,\quad f\in \h.$
\begin{thm}
The controlled $*$-$K$-fusion frame operator $S_{CC'}$ is bounded, positive, self-adjoint, and \begin{equation}
\Vert A^{-1}\Vert^{-2}\Vert K\Vert^2\leq \Vert S\Vert\leq\Vert B\Vert^2.
\end{equation}
\end{thm}
\begin{proof}
Using the definition (\ref{def: c*kff}), we have for all $f\in\h$
\begin{eqnarray*}
\langle S_{CC'}f,f\rangle_{\am} &=& \langle \sum_{i\in I} w_i^2 C'^{*}\pi_{W_i}C(f),f \rangle_{\am}\\
&=&\sum_{i\in I} w_i^2\langle \pi_{W_i}C(f),\pi_{W_i}C'(f) \rangle_{\am}\\
&=&\sum_{i\in I} w_i^2\langle f,C^*\pi_{W_i}C'(f) \rangle_{\am}\\
&=&\langle f,\sum_{i\in I} w_i^2 C'^{*}\pi_{W_i}C(f)\rangle_{\am}\\
&=&\langle f,S_{CC'}f\rangle_{\am}.
\end{eqnarray*}
Then, $S$ is a self adjoint operator.
The inequality (\ref{def: c*kff}) means that for all $f\in\h$, we have \begin{align*}
A\langle K^*f,K^*f\rangle_{\am} A^*\leq \langle S_{CC'}f,f\rangle _{\am}\leq A\langle f,f\rangle_{\am} A^*.
\end{align*} 
It follows  \begin{align*}
\Vert A^{-1}\Vert^{-2}\Vert \langle K K^*f,f\rangle_{\am}\Vert \leq \Vert\langle Sf,f\rangle_{\am}\Vert\leq\Vert B \Vert^2 \Vert\langle f,f\rangle_{\am}\Vert.
\end{align*}
If we take supremum on all $f \in \h$, where $\Vert f\Vert\leq 1$, then $\Vert A_{CC'}^{-1}\Vert^{-2}\Vert K\Vert^2\leq \Vert S_{CC'}\Vert\leq\Vert B_{CC'}\Vert^2$.
\end{proof}
\begin{prop} Let $K,M\in End^*_{\mathcal{A}}(\mathcal{H})$, and $C,C'\in GL(\h)$, The family $\{W_i,w_i\}_{i\in I}$ is a $(C,C')$-controlled $*$-$K$-fusion frame for $\h$ with bounds $A$ and $B$. If $Im(M)\subset  Im(K)$ with $\overline{Im(K^*)}$ is orthogonally complemented, then $\{W_i,w_i\}_{i\in I}$ is a $(C,C')$-controlled $*$-$M$-fusion frame for $\h$ with bounds $\dfrac{A}{\lambda'}$ and $B$.
\end{prop}
\begin{proof}
Since $Im(M)\subset  Im(K)$, from theorem (\ref{lem: ortocomplope}), there exists some $\lambda'>0$ such that $MM^*\leq \lambda' KK^*$.\\ So we have\begin{align*}
\langle MM^*f,f\rangle_{\am}\leq \lambda'\langle KK^*f,f\rangle_{\am}.
\end{align*}
Multipling the above inequality by $A$ and $A^*$, we get \begin{align*}
(\dfrac{A}{\sqrt{\lambda'}})\langle MM^*f,f\rangle_{\am} 
(\dfrac{A}{\sqrt{\lambda'}})^* \leq A\langle KK^*f,f\rangle_{\am} A^*,
\end{align*}
from , we have
Therefore, by combining  with (\ref{def: c*kff}), we conclude that $\{W_i,w_i\}_{i\in I}$ is a $(C,C')$-controlled $*$-$M$-fusion frame for $\h$ with lower and upper bounds $\dfrac{A}{\sqrt{\lambda'}}$ and $B$.
\end{proof}
We further observe that the intersection of the element of a frame of subspaces with a subspases still lives a frame of subspaces for a smaller  submodule of a Hilbert $C^*$-module.\\
\textbf{\textit{Remark}:} If $\{W_i,w_i\}_{i\in I}$ be a $(C,C')$-controlled $*$-$K$-fusion frame for $\h$ with bounds $A$ and $B$ and $V$ be an orthogonally complemented submodule of  $\h$. Then the family  $\{W_i\cap V,w_i\}_{i\in I}$ is a $(C,C')$-controlled $*$-$K$-fusion frame for $\h$ with bounds $A$ and $B$.

\begin{thm} Let $(\h, \am, \langle.,.\rangle_{\am})$ and  $(\h, \mathcal{B}, \langle.,.\rangle_{\mathcal{B}})$ be two Hilbert $C^*$-modules and let $\varphi:\, \, \am\longrightarrow \mathcal{B}$ an $*$-homomorphism and $\Theta$ be a map on $\h$ such that $\langle \Theta x,\Theta y\rangle_{\am}=\varphi(\langle x,y \rangle_{\mathcal{B}})$, for all $x,y\in \h$. Also suppose that the family $\{W_i,w_i\}_{i\in I}$ is a $(C,C')$-controlled $*$-$K$-fusion frame for $\h$ with controlled $*$-$K$-frame operator $S_{\am}$ and lower and upper $*$-$K$-frame bounds A , B respectively. If $\Theta$ is surjective,  $\Theta K^*=K^*\Theta$, $\Theta\pi_{W_i}C=\pi_{W_i}C \Theta $ and $\Theta\pi_{W_i}C'=\pi_{W_i}C'\Theta $, for each $i$ in $I$. Then $\{W_i,\varphi(w_i)\}_{i\in I}$  is a $(C,C')$-controlled $*$-$K$-fusion frame for $(\h, \mathcal{B}, \langle.,.\rangle_{\mathcal{B}})$. with controlled $*$-$K$-frame operator $S_{\mathcal{B}}$ and lower and upper $*$-$K$-frame bounds $\varphi(A)$, $\varphi(B)$ respectively, and $\langle S_{\mathcal{B}} \Theta f,\Theta g \rangle_{\mathcal{B}}=\varphi(\langle S_{\am}f, g\rangle_{\am}).$
\end{thm}
\begin{proof} Let $g\in \h$, then there exists $f\in \h$ such that $\Theta f=g$, by the definition of the $(C,C')$-controlled $*$-$K$-fusion frame for $\h$ \begin{align*}
A \langle K^* f, K^* f\rangle_{\am}  A^* \leq\sum_{i\in I} w_i^2\langle\pi_{W_i}C f,\pi_{W_i}C' f\rangle_{\am}\leq B \langle f,f\rangle_{\am}  B^*, 
\end{align*}
Then,  \begin{align*}
\varphi(A \langle K^* f, K^* f\rangle_{\am}  A^*) \leq\varphi(\sum_{i\in I} w_i^2\langle\pi_{W_i}C f,\pi_{W_i}C' f\rangle_{\am})\leq \varphi(B \langle f,f\rangle_{\am}  B^*), 
\end{align*}
Since, $\varphi$ is an $*$-homomorphism, we have
\begin{align*}
\varphi(A)\varphi( \langle K^* f, K^* f\rangle_{\am})\varphi(  A^*) \leq \sum_{i\in I} \varphi( w_i^2)\varphi(\langle\pi_{W_i}C f,\pi_{W_i}C' f\rangle_{\am})\leq \varphi(B)\varphi( \langle f,f\rangle_{\am})\varphi(  B^*), 
\end{align*}
By the relation betwen $\Theta$ and $\varphi$ 
\begin{align*}
\varphi(A) \langle  \Theta K^* f, \Theta K^* f\rangle_{\mathcal{B}}\varphi(  A^*) \leq \sum_{i\in I} \varphi( w_i^2)\langle\Theta \pi_{W_i}C f,\Theta \pi_{W_i}C' f\rangle_{\mathcal{B}})\leq \varphi(B) \langle \Theta f,\Theta f\rangle_{\mathcal{B}})\varphi(  B^*), 
\end{align*}
\begin{align*}
\varphi(A) \langle   K^* \Theta f, K^* \Theta f\rangle_{\mathcal{B}}\varphi(  A^*) \leq \sum_{i\in I} \varphi( w_i^2)\langle\Theta \pi_{W_i}C f,\Theta \pi_{W_i}C' f\rangle_{\mathcal{B}})\leq \varphi(B) \langle \Theta f,\Theta f\rangle_{\mathcal{B}})\varphi(  B^*), 
\end{align*}
On the other hand, we have \begin{eqnarray*}
\varphi(\langle S_{\am}f, g\rangle_{\am})&=& \varphi(\langle\sum_{i\in I} w_i^2 \pi_{W_i}C f,\pi_{W_i}C' g> \rangle_{\am})\\
&=&\sum_{i\in I} \varphi( w_i^2)\langle \Theta \pi_{W_i}C f,\Theta \pi_{W_i}C'g \rangle_{\mathcal{B}}\\
&=&\sum_{i\in I} \varphi( w_i^2)\langle \pi_{W_i}C\Theta f, \pi_{W_i}C'\Theta g \rangle_{\mathcal{B}}\\
&=&\sum_{i\in I} \varphi( w_i^2)\langle C'^*\pi_{W_i}C\Theta f, \Theta g \rangle_{\mathcal{B}}\\
&=&\langle \sum_{i\in I} \varphi( w_i^2) C'^*\pi_{W_i}C \Theta f,\Theta g \rangle_{\mathcal{B}}\\
&=&\langle S_{\mathcal{B}} \Theta f,\Theta g \rangle_{\mathcal{B}}.
\end{eqnarray*}
Which completes the proof.
\end{proof}

The following result will provide explicit estimates for the controlled $*$-$K$-fusion frame of a controlled $*$-$K$-fusion frame of the form $\{UW_i,w_i\}_{i\in I}$ in Hilbert $C^*$-modules.
\begin{thm} Let $\{W_i,w_i\}_{i\in I}$ be a $(C,C')$-controlled $*$-$K$-fusion frame for $\h$ with bounds $A$ and $B$, and with controlled frame operator $S_{CC'}$ and $ U^{-1}$ commutes with $C$, $C'$ and $K^*$. If $U$ is an invertible elemnet of $End_{\am}^*({\mathcal{H}})$ such that both are uniformally bounded and for each $i\in I$, $U^*U(W_i)\subseteq W_i$, then, $\{UW_i,w_i\}_{i\in I}$ is a $(C,C')$-controlled $*$-$K$-fusion frame for $\h$. In this case, $S_{CC'}U := US_{CC'}U^{-1}$.
\end{thm}
\begin{proof} Since $W_i$ for each $i\in I$ , is a closed submodule of $\h$, so each $UW_i$ is also aclosed submodule of $\h$. Since $U$ is surjective, for each $f\in\h$, there exist a unique $g\in\h$ such that $f = Ug$. On the other hand, for each $i\in I$, $W_i$ is an orthogonally complemented submodule, hence there exist unique elements $x\in W_i$ and $y\in W_i^{\perp}$ such that $g=x+y$, thus $f=Ux+Uy$, and since $UW_i\cap UW_i^\perp=\{0\}$, for each $i\in I$, we have \begin{align*}
\h=UW_i\oplus U(W_i^\perp).
\end{align*}
On other hand, since $U(W_i^{\perp})=(U W_i)^{\perp}$, then for each $i\in I$  implies that $UW_i$ is orthogonally complemented  \begin{eqnarray*}
\sum_{i\in I} w_i^2 \langle \pi_{UW_i}C f,\pi_{UW_i}C'f\rangle_{\am} &=& \sum_{i\in I} w_i^2 \langle U\pi_{W_i}U^{-1}C f,U\pi_{W_i} U^{-1}C'f\rangle_{\am}\\ &\leq & \Vert U\Vert^2 \sum_{i\in I} w_i^2 \langle \pi_{W_i}U^{-1}C f,\pi_{W_i} U^{-1}C'f\rangle_{\am}\\
&=& \Vert U\Vert^2 \sum_{i\in I} w_i^2 \langle \pi_{W_i}C U^{-1} f,\pi_{W_i} C' U^{-1}f\rangle_{\am} \\
&\leq& \Vert U\Vert^2  B \langle U^{-1} f,U^{-1}f\rangle_{\am} B^*\\
&\leq& (\Vert U\Vert \Vert U^{-1}\Vert B) \langle  f,f\rangle_{\am} (\Vert U\Vert \Vert U^{-1}\Vert B)^*.
\end{eqnarray*}
On other hand,
\begin{eqnarray*}
\sum_{i\in I} w_i^2 \langle \pi_{UW_i}C f,\pi_{UW_i}C'f\rangle_{\am} &=& \sum_{i\in I} w_i^2 \langle U\pi_{W_i}U^{-1}C f,U\pi_{W_i} U^{-1}C'f\rangle_{\am}\\
&\geq & A \Vert U\Vert^{-2}\langle K^*U^{-1}f,K^*T^{-1}f \rangle_{\am} A^*  \\
&\geq & A \Vert U\Vert^{-2}\langle U^{-1}K^*f,T^{-1}K^*f \rangle_{\am} A^*\\
&\geq &  (\Vert U\Vert \Vert U^{-1}\Vert A)\langle K^*f,K^*f \rangle_{\am} (\Vert U\Vert \Vert U^{-1}\Vert A)^*. 
\end{eqnarray*}
That shows that $\{UW_i,w_i\}_{i\in I}$ is a $(C,C')$-controlled $K$-fusion frame for $\h$.
\end{proof}

\begin{cor}  Let $\{W_i,w_i\}_{i\in I}$ is a $(C,C')$-controlled $*$-$K$-fusion frame for $\h$ and $K\in End_{\am}^{*}(\h)$ be a surjective operator with controlled $*$-$K$-fusion frame operator $S$. Then $\{W_iS^{-1},w_i\}_{i\in I}$ is a $(C,C')$-controlled $*$-$T$-fusion frame for $\h$. 
\end{cor}
We conclude this section with the results showing that the removal of an element from a controlled $*$-$K$-fusion frame can leave again a controlled $*$-$K$-fusion frame.
\begin{thm} Let $\{W_i,w_i\}_{i\in I}$  be a $(C,C')$-controlled $*$-$K$-fusion frame for $\h$ with bounds $A, B$. Let $i_{0}\in I$ be given. Suppose that $K$ has closed range and $Ran(K)$ is orthogonally complemented. We have the following results:
\begin{itemize}
\item[i)] If $ \oplus_{i\in {I\setminus \{i_0\}}} \mathcal{W}_i\subset Ran(K)$ and $Ran(C'^*\pi_{W_{i_0}}C) \perp Ran(T)$, then $\{W_i,w_i\}_{i\in I\setminus \{i_0\}}$ is a $(C,C')$-controlled $K$-fusion frame for $\h$.
\item[ii)] If  $ \oplus_{i\in {I\setminus \{i_0\}}} \mathcal{W}_i\subset Ran(K)$ and  $A- \Vert C'^*\pi_{W_{i_0}}C \Vert^2\Vert K^\dagger\Vert^2 \neq 0_{\am}$, then $\{W_i,w_i\}_{i\in I\setminus \{i_0\}}$ is a $(C,C')$-controlled $*$-$K$-fusion frame for $\h$, where $K^\dagger$ is the Moore–Penrose inverse of $K$.
\end{itemize}
\end{thm}
\begin{proof}
\item[i)] For any $f\in Ran (K)$, and since $h\in Ker(C'^*\pi_{W_{i_0}}C)$, we have 

 \begin{eqnarray*}
A\langle K^*f,K^*f\rangle_{\am} A^* &\leq &  \sum_{i\in I} w_i^2\langle \pi_{W_i}C f,\pi_{W_i}C'f \rangle_{\am} \\
&=& \sum_{i\in I \setminus \{i_0\}} w_i^2\langle \pi_{W_i} C f,\pi_{W_i}C'f\rangle_{\am}.
\end{eqnarray*}
 Since $Ran(K)$ is orthogonally complemented, every $f\in \h$ has a decomposition as $f=f_1+f_2$ where $f_1\in Ran(K)$ and $f_2\in (Ran(K))^{\perp}$, a direct calculation shows $K^*f_2=0.$ We have 
 also \begin{align*}
A\langle K^*f,K^*f\rangle_{\am} A^*= A\langle K^*f_1,K^*f_1\rangle_{\am} A^* \leq  \sum_{i\in I \setminus \{i_0\}} w_i^2\langle \pi_{W_i} C f,\pi_{W_i}C'f\rangle_{\am}.
\end{align*}
\item[ii)] Since $K$ has closed range every $h\in Ran(K)$ can be expressed as $h=KK^\dagger h=(KK^\dagger)^*h=(K^\dagger)^*K^*h.$
\begin{align*}
\langle h,h\rangle_{\am}=\langle(K^\dagger)^*K^*h,(K^\dagger)^*K^*h\rangle _{\am}\leq \Vert (K^\dagger)^*\Vert^2 \langle K^*h, K^*h\rangle_{\am}.
\end{align*}
For all $f\in\h$, again appling tha fact that $f=f_1+f_2$ with $f_1\in Ran(K)$ and $f_2\in (Ran(K))^{\perp}$, we have 
\begin{eqnarray*}
\sum_{i\in I \setminus \{i_0\}}w_i^2\langle \pi_{W_i} C f,\pi_{W_i}C'f\rangle_{\am} &=& \sum_{i\in I \setminus \{i_0\}}w_i^2\langle \pi_{W_i} C f_1,\pi_{W_i}C'f_1\rangle_{\am}\\
&=&\sum_{i\in I }w_i^2\langle \pi_{W_i} C f_1,\pi_{W_i}C'f_1\rangle_{\am}- w_{i_0}^2\langle \pi_{W_{i_0}} C f_1,\pi_{W_{i_0}}C'f_1\rangle_{\am}\\
&\geq & A \langle K^*f_1,K^*f_1\rangle_{\am} A^* - w_{i_0}^2\Vert (C'^*\pi_{W_{i_0}} C)^{\frac{1}{2}}\Vert^2\langle f_1,f_1\rangle_{\am}\\
&\geq & A \langle K^*f_1,K^*f_1\rangle_{\am} A^*- w_{i_0}^2\Vert (C'^*\pi_{W_{i_0}} C)^{\frac{1}{2}}\Vert^2 \Vert (K^\dagger)^*\Vert^2 \langle K^*f_1, K^*f_1\rangle_{\am}\\
&=&  A \langle K^*f_1,K^*f_1\rangle_{\am} A^*- (w_{i_0}\Vert (C'^*\pi_{W_{i_0}} C)^{\frac{1}{2}}\\&&\Vert \Vert (K^\dagger)^\star\Vert) \langle K^*f_1, K^*f_1\rangle_{\am}(w_{i_0}\Vert (C'^*\pi_{W_{i_0}} C)^{\frac{1}{2}}\Vert \Vert (K^\dagger)^*\Vert)^*\\
&=&  (A- w_{i_0}\Vert (C'^*\pi_{W_{i_0}} C)^{\frac{1}{2}}\Vert \Vert (K^\dagger)^*\Vert)\\&&\langle K^*f_1, K^*f_1\rangle_{\am} (A-w_{i_0}\Vert (C'^*\pi_{W_{i_0}} C)^{\frac{1}{2}}\Vert \Vert (K^\dagger)^*\Vert)^*.
\end{eqnarray*}
\end{proof}

\begin{thm} Let $K_1,K_2\in End_{\am}^{*}(\h)$ such that $ Im(K^*_1)\perp Im(K^*_1)$. If $\{W_i,w_i\}_{i\in I}$ is a $(C,C')$-controlled  $*$-$K_1$-fusion frame for $\h$ as well as a $*$-$K_2$-fusion frame for $\h$, $\alpha$ and $\beta$ are  scalars, then $\{W_i,w_i\}_{i\in I}$ is a  $(C,C')$-controlled   $*$-$(\alpha K_1+\beta K_2)$-fusion frame and a $(C,C')$-controlled   $*$-$(K_1K_2)$-fusion frame for $\h$. 
\end{thm}
\begin{proof} By supposition, there exist strictly nonzero elements $A_1,A_2$ and $B_2,B_2$ in $\am$, such that for any  $f\in\h$, we have 
\begin{equation} 
A_1 \langle K^*_1 f, K^*_1 f\rangle_{\am} A_1^*\leq\sum_{i\in I} w_i^2\langle\pi_{W_i}C f,\pi_{W_i}C' f\rangle_{\am} \leq B_1 \langle f,f\rangle_{\am} B_1^*, 
\end{equation}
\begin{equation} 
A_2 \langle K^*_2 f, K^*_2 f\rangle_{\am} A_2^*\leq\sum_{i\in I} w_i^2\langle\pi_{W_i}C f,\pi_{W_i}C' f\rangle _{\am}\leq B_2 \langle f,f\rangle_{\am} B_2^*, 
\end{equation}
Now, for any $f\in\h$, we have 
\begin{eqnarray*}
\langle(\alpha K_1+\beta K_2)^*f,(\alpha K_1+\beta K_2)^*f\rangle_{\am}&=& \vert \alpha \vert^2\langle K^*_1 f, K^*_1 f\rangle_{\am} + \overline{\alpha}\beta \langle K^*_1 f, K^*_2 f\rangle_{\am}\\&+&\alpha\overline{\beta} \langle K^*_2 f, K^*_1 f\rangle_{\am}+\vert \beta \vert^2 \langle K^*_2 f, K^*_2 f\rangle_{\am}.
\end{eqnarray*}
Since $ Im(K^*_1)\perp Im(K^*_1)$, then 
\begin{align*}
\langle(\alpha K_1+\beta K_2)^*f,(\alpha K_1+\beta K_2)^*f\rangle_{\am}&=& \vert \alpha \vert^2\langle K^*_1 f, K^*_1 f\rangle_{\am} + \vert \beta \vert^2 \langle K^*_2 f, K^*_2 f\rangle_{\am}.
\end{align*}
Let $B=\left(|\alpha|^{2} A_{2} A_{2}^{*}+|\beta|^{2} A_{1} A_{1}^{*}+1\right)$.\\
Therefore, for each $f \in \h$, we have
$$
\begin{aligned}
&A_{1} A_{2} B^{-1}\left\langle\left(\alpha K_{1}+\beta K_{2}\right)^{*} f,\left(\alpha K_{1}+\beta K_{2}\right)^{*} f\right\rangle_{A A}\left(A_{1} A_{2}\right)^{*} \\
&\quad=A_{1} A_{2} B^{-1 / 2}\left\langle\left(\alpha K_{1}+\beta K_{2}\right)^{*} f,\left(\alpha K_{1}+\beta K_{2}\right)^{*} f\right\rangle_{A}\left(A_{1} A_{2} B^{-1 / 2}\right) \\
&\quad \leq \sum_{i\in I} w_i^2\langle\pi_{W_i}C f,\pi_{W_i}C' f\rangle_{\am} \leq\left(B_{1}+B_{2}\right)\langle f, f\rangle_{A}\left(B_{1}+B_{2}\right)^{*}.
\end{aligned}
$$
This shows that $\{W_i,w_i\}_{i\in I}$ is a controlled $*$-$\left(\alpha K_{1}+\beta K_{2}\right)$-frame for $\h$. Furthermore, for every $f \in \h$, we have
$$
\begin{aligned}
\left\langle\left(K_{1} K_{2}\right)^{*} f,\left(K_{1} K_{2}\right)^{*} f\right\rangle_{A} &=\left\langle K_{2}^{*} K_{1}^{*} f, K_{2}^{*} K_{1}^{*} f\right\rangle_{A} \\
& \leq\left\|K_{2}^{*}\right\|^{2}\left\langle K_{1}^{*} f, K_{1}^{*} f\right\rangle_{A}.
\end{aligned}
$$

By supposition, $\{W_i,w_i\}_{i\in I}$ is a $(C,C')$-controlled  $*$-$K_1$-fusion frame for $\h$, for each $f \in \h$, we have 
$$
\begin{aligned}
&A_{1}\left\|K_{2}^{*}\right\|^{-2}\left\langle\left(K_{1} K_{2}\right)^{*} f,\left(K_{1} K_{2}\right)^{*} f\right\rangle_{A} A_{1}^{*} \\
&\quad \leq \sum_{i\in I} w_i^2\langle\pi_{W_i}C f,\pi_{W_i}C' f\rangle_{\am}\leq B_{1}\langle f, f\rangle_{A} B_{1}^{*}.
\end{aligned}
$$
It follows that,
$$
\begin{aligned}
&A_{1}\left\|K_{2}^{*}\right\|^{-1}\left\langle\left(K_{1} K_{2}\right)^{*} f,\left(K_{1} K_{2}\right)^{*} f\right\rangle_{A}\left(A_{1}\left\|K_{2}^{*}\right\|^{-1}\right)^{*} \\
&\quad \leq \sum_{i\in I} w_i^2\langle\pi_{W_i}C f,\pi_{W_i}C' f\rangle_{\am}\leq B_{1}\langle f, f\rangle_{\am} B_{1}^{*}.
\end{aligned}
$$
Therefore, $\{W_i,w\}_{i\in I}$ a  $\left(C, C^{\prime}\right)$-controlled $*$-$\left(K_{1} K_{2}\right)$-fusion frame for $\h$.

\end{proof}
\begin{cor} Let $K\in End_{\am}^{*}(\h)$, if $\{W_i,w\}_{i\in I}$ is a $(C,C')$-controlled $*$-$K$-fusion frame for $\h$. Then for any operator $T$ in the subalgebra generated by $\mathcal{K}$, the family $\{W_i,w_i\}_{i\in I}$ is a $(C,C')$-controlled $*$-$T$-fusion frame for $\h$.  
\end{cor}
\begin{prop} Let $K\in End_{\am}^{*}(\h)$ and $\{W_i,w_i\}_{i\in I}$ be a $(C,C')$-controlled $*$-$K$-fusion frame for $\h$. Suppose that $Im(K)$ is closed, if $T\in End_{\am}^{*}(\h)$ with $Im(T)\subset Im(K)$, then $\{W_i,w_i\}_{i\in I}$ is a $(C,C')$-controlled $*$-$T$-fusion frame for $\h$. 
\end{prop}

\begin{thm} Let $K\in End_{\am}^{*}(\h)$ and $C,C'\in GL^+ (\h)$ such that $ CC'=C'C$. Suppose that $K$ and $S$ commute with $C$ and $C'$. Then $\{W_i,w_i\}_{i\in I}$ is  a $*$-$K$-fusion frame sequence in  $\h$ if and only if  $\{W_i,w_i\}_{i\in I}$ is a controlled $*$-$K$-fusion frame for  $\h$.\\
Where $S$ is the $*$-$K$-fusion frame operator, defined by \begin{align}
Sf=\sum_{i\in I }w_i^2 \pi_{W_i}f,\quad f\in\h.
\end{align}
 \end{thm}
\begin{proof} Suppose that $\{W_i,w_i\}_{i\in I}$ is  a $*$-$K$-fusion frame sequence in  $\h$ with bounds $A$ and $B$, By definition, for all $f\in \h$, we have 
\begin{align*}
A \langle K K^* f, f\rangle_{\am} A^*\leq \langle Sf, f\rangle_{\am} \leq B \langle f,f\rangle_{\am} B^*, 
\end{align*} 
then, 
\begin{align}\label{th: eoilecontrolled1}
A  K K^*  A^*\leq  S \leq B Id_{\am} B^*, 
\end{align}
\begin{align*}
m \leq  C \leq  M,
\end{align*} 
\begin{align*}
m' \leq C \leq  M'.
\end{align*}
since, \begin{align*}
\langle CSf, f\rangle_{\am}=\langle SCf, f\rangle_{\am}=\langle f, SCf\rangle_{\am}.
\end{align*}
we have 
\begin{align}\label{th: eoilecontrolled2}
m A  K K^*  A^*\leq  S \leq M B Id_{\am} B^*
\end{align}
Therfore, for all $f\in\h$ \begin{align*}
mm'A \langle K^* f, K^*f\rangle_{\am} A^*\leq \sum_{i\in I }w_i^2 \langle \pi_{W_i}Cf, \pi_{W_i}C'f\rangle_{\am}\leq MM'B\langle f,f\rangle B^*
\end{align*} 
Hence, $\{W_i,w_i\}_{i\in I}$ is a $(C,C')$-controlled $*$-$K$-fusion frame for $\h$.\\
Conversely, assume that $\{W_i,w_i\}_{i\in I}$ is a $(C,C')$-controlled $*$-$K$-fusion frame for $\h$ with bounds $A_{CC'}$ and $B_{CC'}$

Take for all $f\in \h$, Firstly, let us prove that $\{W_i,w_i\}_{i\in I}$ is a Bessel $*$-$K$-fusion sequence for $\h$
\begin{eqnarray*}
\sum_{i\in I }w_i^2 \langle \pi_{W_i} f, \pi_{W_i} f\rangle_{\am} &=&\langle Sf,f\rangle_{\am}\\
&=&\langle (CC')^{\frac{1}{2}} (CC')^{-\frac{1}{2}}Sf,(CC')^{\frac{1}{2}} (CC')^{-\frac{1}{2}}f\rangle_{\am} \\
&=& \langle(CC')^{1}(CC')^{-\frac{1}{2}}  Sf,(CC')^{-\frac{1}{2}}f\rangle_{\am}\\
&=&\langle CC'S (CC')^{-\frac{1}{2}} f,(CC')^{-\frac{1}{2}} f\rangle_{\am}\\
&\leq & B\Vert (CC')^{-\frac{1}{2}} \Vert^2 \langle f, f\rangle_{\am}B^*\\
&\leq & (B\Vert (CC')^{-\frac{1}{2}} \Vert) \langle f, f\rangle_{\am}(B\Vert (CC')^{-\frac{1}{2}} \Vert)^*.
\end{eqnarray*}

Now, let us establish for $\{W_i,w_i\}_{i\in I}$ the left-hand side of inequality (\ref{def: c*kff}). For all $f\in\h$, we have
\begin{eqnarray*}
A_{CC'}\langle  K^* f,K^*f\rangle_{\am} A^*_{CC'}&=& A_{CC'}\langle (CC')^{\frac{1}{2}} (CC')^{-\frac{1}{2}}K^* f,(CC')^{\frac{1}{2}} (CC')^{-\frac{1}{2}} K^*f\rangle_{\am} A^*_{CC'}\\
&=& A_{CC'}\langle (CC')^{\frac{1}{2}} K^* (CC')^{-\frac{1}{2}} f,(CC')^{\frac{1}{2}}K^* (CC')^{-\frac{1}{2}}f\rangle_{\am} A^*_{CC'} \\
&\leq &\Vert (CC')^{-\frac{1}{2}} \Vert^2 \sum_{i\in I }w_i^2 \langle \pi_{W_i} C (CC')^{-\frac{1}{2}}f, \pi_{W_i}C'(CC')^{-\frac{1}{2}} f\rangle_{\am} \\
&\leq &\Vert (CC')^{-\frac{1}{2}} \Vert^2  \langle S C (CC')^{-\frac{1}{2}}f, C'(CC')^{-\frac{1}{2}} f\rangle_{\am} \\
&=&\Vert (CC')^{-\frac{1}{2}} \Vert^2  \langle S C (CC')^{-\frac{1}{2}}f, C'(CC')^{-\frac{1}{2}} f\rangle_{\am} \\
&=&\Vert (CC')^{-\frac{1}{2}} \Vert^2  \langle S  C^{\frac{1}{2}}C'^{-\frac{1}{2}}f, C'^{\frac{1}{2}} C^{-\frac{1}{2}} f\rangle_{\am} \\
&=&\Vert (CC')^{-\frac{1}{2}} \Vert^2  \langle C^{-\frac{1}{2}}C'^{\frac{1}{2}} S  C^{\frac{1}{2}}C'^{-\frac{1}{2}}f,f\rangle_{\am}\\
&=&\Vert (CC')^{-\frac{1}{2}} \Vert^2 \langle Sf,f\rangle_{\am}.
\end{eqnarray*}
Thus, 
\begin{align}
 A_{CC'} \Vert (CC')^{-\frac{1}{2}} \Vert^{-2}\langle  K^* f,K^*f\rangle_{\am} A^*_{CC'} \leq \langle Sf,f\rangle_{\am}.
\end{align}
\end{proof}
\subsection{Erasures  of submodules}
Our main result in this subsection provides sufficient conditions on the weights for a subspace to be deleted yet still leave a controlled $K$-fusion frame. in fact, the following result is a generalization of Theorem $(3.2)$ in Ref. \cite{CK08} to Hilbert $C^*$-modules for a controlled $K$-fusion frame. Assume that  $\{W_i\}_{i\in I}$ be a sequence of orthogonally complemented submodules of $\h$. Then the Hilbert $C^*$-submodule of $\h$ generated by $\{W_i\}_{i\in I}$ is denoted by $ <\{W_i\}_{i\in I\setminus \{J\}}>$.
\begin{thm}
Let  $\{W_i,w_i\}_{i\in I}$ be a $(C,C)$-controlled $K$-fusion frame for $\h$ with bounds $A$ and $B$ and $J\subset I$ such that $\sum_{i\in J}w_i^2$ is convegent in $\am$, assume that $\pi_{W_i}$ and $C^{-1}$ commute each other. If 
\begin{itemize}
\item[i)] If $B \Vert C^{-1}\Vert ^2 .1_{\am} < \sum_{i\in J}w_i^2$, then $\cap_{i\in J} W_i=\{0\}$.
\item[ii)] If $\Vert (K^\dagger)^*\Vert^2\sum_{i\in J}w_i^2< A .1_{\am} $, then $\{W_i,w_i\}_{i\in I\setminus \{J\}}$ is a $(C,C^{-1})$-controlled $K$-fusion frame for $\h$.
\end{itemize}
\end{thm}
\begin{proof}
\begin{itemize}
\item[i)] Let $f\in \cap_{i\in J} W_i$, since $\pi_{W_i}f= f$ for each $i\in J$ and $\sum_{i\in J}w_i^2\in \mathcal{Z}(\am)$, we have \begin{eqnarray*}
B \Vert C^{-1}\Vert ^2\langle  f,f\rangle_{\am}
&\leq & \sum_{i\in J}w_i^2\langle \pi_{W_i} C^{-1}C f,\pi_{W_i} C^{-1}C f\rangle _{\am}\\
&\leq & \Vert C^{-1}\Vert ^2 \sum_{i\in I}w_i^2\langle \pi_{W_i} C f,\pi_{W_i}Cf \rangle_{\am} \\
&\leq & B\Vert C^{-1}\Vert ^2\langle f,f\rangle_{\am}.
\end{eqnarray*}
 $( \sum_{i\in J}w_i^2-B \Vert C^{-1}\Vert ^2.1_{\am})f=0$, and since $\sum_{i\in J}w_i^2-B \Vert C^{-1}\Vert ^2.1_{\am}$ is invertible, we get $\langle f,f\rangle=0$ consequently $f=0$.
\item[ii)] Since $\sum_{i \in I} \omega_{i}^{2}\left\langle\pi_{W_{i}}Cf, \pi_{W_{i}}Cf\right\rangle_{\am}$ converges in norm for each $f\in \h$, we obtain that $\sum_{i \in I-J} \omega_{i}^{2}\left\langle\pi_{W_{i}} Cf, \pi_{W_{i}} Cf\right\rangle_{\am}$ converges in norm and also,
\begin{eqnarray*}
\left\|\left(A .1_{\am}- \Vert (K^\dagger)^*\Vert^2\left(\sum_{i \in J} \omega_{i}^{2}\right)\right)^{-1}\right\|^{-1}\langle f, f\rangle_{\am} &\leq&\left(A \cdot 1_{\am}- \Vert (K^\dagger)^*\Vert^2 \left(\sum_{j \in J} \omega_{i}^{2}\right)\right)\langle f,f\rangle_{\am}\\
&\leq &\Vert (K^\dagger)^*\Vert^2\sum_{i \in I-J} \omega_{i}^{2}\left\langle\pi_{W_{i}} Cf, \pi_{W_{i}} Cf\right\rangle_{\am}.
\end{eqnarray*}
Therefore, $ \mathcal{W}$ is a standard  $(C,C)$-controlled $K$-fusion frame.
\end{itemize}
\end{proof}
\section{Perturbation of controlled $K$-fusion frame}
In this section, we introduce a new perturbation for  controlled $K$-fusion frames in Hilbert $C^*$-modules which controlled $K$-fusion frames are stable under it.
\begin{defi}
Let $\left\{W_{i}\right\}_{i \in I}$ and $\left\{V_{i}\right\}_{i \in I}$ be two families of orthogonally complemented submodules of $\h .$ Let $\left\{v_{i}\right\}_{i \in I}$ be a sequence of weights, $\left\{a_{i}\right\}_{i \in I} \in \ell^{2}(I, \am)$ and $a_{1}, a_{2} \in \am$ with $\left\|a_{1}\right\|_{2}\left\|a_{2}\right\|<1$. If
\begin{align}
\label{def: perturbation}
\left|v_{i} C'^*\left(\pi_{W_{i}}-\pi_{V_{i}}\right) C f\right|^{2} \leq\left|a_{1}\mid v_{i} C'^* \pi_{W_{i}} C f\mid+a_{2}\mid v_{i}  C'^*\pi_{V_{i}} C f\mid+\mid f\mid a_{i}\right|^{2},
\end{align}
for each $ f\in \h $ and $i \in I$, then we say that $\left\{\left(V_{i}, v_{i}\right)\right\}_{i \in I}$ is an $\left(a_{1}, a_{2},\left\{a_{i}\right\}_{i \in I}\right)$- controlled perturbation of $\left\{\left(W_{i}, v_{i}\right)\right\}_{i \in I}$.

\end{defi}
\begin{thm}  \label{thm: perturbation}
Let $K\in End^*_{\am}(\h)$ be a closed range, suppose that $W=\left\{\left(W_{i}, v_{i}\right)\right\}_{i \in I}$ is a standard $(C,C')$-conrolled $K$-fusion frame with lower bound $A$ and $V=\left\{\left(V_{i}, v_{i}\right)\right\}_{i \in I}$ is an $\left(a_{1}, a_{2},\left\{a_{i}\right\}_{i \in I}\right)$-controlled perturbation of $W .$ If $\left\|\left\{a_{i}\right\}_{i \in I}\right\|_{2}<$ $\left(1-\left\|a_{1}\right\|\right) \sqrt{A}$, then $V$ is a standard $(C,C')$-controlled $K$-fusion frame.
\end{thm}
\begin{proof}
Let $J$ be a finite subset of $I$, $f \in \h$ and let $B$ be an upper bound for $\Lambda$. Then from (\ref{def: perturbation}), for each $f \in \h$, we can obtain that
$$
\begin{aligned}
\left\|\left\{v_{i} C'^*\left(\pi_{W_{i}}-\pi_{V_{i}}\right) C f\right\}_{i \in F}\right\|_{l^2(\h)}& \leq\left\|a_{1}\mid v_{i} C'^* \pi_{W_{i}} C f\mid+a_{2}\mid v_{i}  C'^*\pi_{V_{i}} C f\mid+\mid f\mid a_{i}\}_{i \in F}\right\|_{l^2(\h)} \\
& \leq\left\|a_{1}\right\|\left\|\left\{ v_i C'^* \pi_{W_{i}} C f\right\}_{i \in F}\right\|_{l^2(\h)}\\&+\left\|a_{2}\right\|\left\|\left\{ v_i C'^* \pi_{V_{i}} C f\right\}_{i \in F}\right\|_{l^2(\h)}+\|f\|\left\|\left\{a_{i}\right\}_{i \in F}\right\|_{l^2(\h)}.
\end{aligned}.
$$
Now we have

\begin{eqnarray*}
\left\|\left\{ v_i C'^* \pi_{V_{i}} C f\right\}_{i \in F}\right\|_{l^2(\h)} & \leq & \left\|\left\{ v_i C'^* \pi_{W_{i}} C f\right\}_{i \in F}\right\|_{l^2(\h)}+\left\|\left\{\left( v_i C'^* (\pi_{W_{i}}- \pi_{W_{i}}) C f\right) f\right\}_{i \in F}\right\|_{l^2(\h)} \\
& \leq & \left(1+\left\|a_{1}\right\|\right)\left\|\left\{ v_i C'^* \pi_{W_{i}} C f\right\}_{i \in F}\right\|_{l^2(\h)}\\
&+& \left\|a_{2}\right\|\left\|\left\{ v_i C'^* \pi_{V_{i}} C f\right\}_{i \in F}\right\|_{l^2(\h)}+\|f\|\left\|\left\{a_{i}\right\}_{i \in F}\right\|_{l^2(\h)}.
\end{eqnarray*}

Hence\begin{align*}
\left\|\left\{ v_i C'^* \pi_{V_{i}} C f\right\}_{i \in F}\right\|_{l^2(\h)} \leq \frac{1}{1-\left\|a_{2}\right\|}\left((1+\left\|a_{1}\right\|\right)\left\|\left\{ v_i C'^* \pi_{W_{i}} C f\}_{i \in F}\right\|_{l^2(\h)}+\|f\|\left\|\left\{a_{i}\right\}_{i \in F}\right\|_{l^2(\h)}\right)
\end{align*}
Since $F$ is an arbitrary subset of $I$, we have
$$
\left\|\sum_{i \in I}v_i^2\left\langle \pi_{V_{i}}Cf,\pi_{V_{i}}C' f\right\rangle_{\am}\right\|=\left\|\left\{v_i C'^* \pi_{V_{i}} C f\right\}_{i \in I}\right\|_{l^2(\h)}^{2} \leq \frac{1}{\left(1-\left\|a_{2}\right\|\right)^{2}}\left(\left(1+\left\|a_{1}\right\|\right) \sqrt{B}+\left\|\left\{a_{i}\right\}_{i \in I}\right\|_{l^2(\h)}\right)^{2}\|f\|^{2}
$$
Also
\begin{eqnarray*}
\left\|\left\{v_i C'^* \pi_{V_{i}} C f\right\}_{i \in F}\right\|_{l^2(\h)} & \geq &\left\|\left\{ v_i C'^* \pi_{W_{i}} C f\right\}_{i \in F}\right\|_{l^2(\h)}-\left\|\left\{ v_i C'^*\left(\pi_{W_{i}} - \pi_{V_{i}}\right )C f \right\}_{i \in F}\right\|_{l^2(\h)} \\
& \geq &\left(1-\left\|a_{1}\right\|\right)\left\|\left\{v_i C'^* \pi_{W_{i}} C f\right\}_{i \in F}\right\|_{l^2(\h)}\\&-&\left\|a_{2}\right\|\left\|\left\{v_i C'^* \pi_{V_{i}} C f\right\}_{i \in F}\right\|_{l^2(\h)}-\|f\|\left\|\left\{a_{i}\right\}_{i \in F}\right\|_{l^2(\h)},
\end{eqnarray*}
thus
$$
\left\|\left\{v_i C'^* \pi_{V_{i}} C f\right\}_{i \in F}\right\|_{l^2(\h)}\geq \frac{1}{1+\left\|a_{2}\right\|}\left(\left(1-\left\|a_{1}\right\|\right)\left\|\left\{v_i C'^* \pi_{W_{i}} C f\right\}_{i \in F}\right\|_{l^2(\h)}-\|f\|\left\|\left\{a_{i}\right\}_{i \in F}\right\|_{l^2(\h)}\right)
$$
$$
\begin{aligned}
\left\|\sum_{i \in I}\left\langle\pi_{V_{i}}Cf,\pi_{V_{i}}C' f\right\rangle\right\|^{\frac{1}{2}} &=\left\|\left\{ v_i C'^* \pi_{V_{i}} C f\right\}_{i \in I}\right\|_{l^2(\h)} \\
& \geq \frac{1}{1+\left\|a_{2}\right\|}\left(\left(1-\left\|a_{1}\right\|\right)\left\|\left\{ v_i C'^* \pi_{W_{i}} C f\right\}_{i \in I}\right\|_{l^2(\h)}-\|f\|\left\|\left\{a_{i}\right\}_{i \in I}\right\|_{l^2(\h)}\right) \\
& \geq \frac{1}{1+\left\|a_{2}\right\|}\left(\left(1-\left\|a_{1}\right\|\right) \sqrt{A}- \|(K^\dagger)\|^{-2}\left\|\left\{a_{i}\right\}_{i \in I}\right\|_{l^2(\h)}\right)\|K^*f\|
\end{aligned}
$$
Thus by Theorem $(3.1)$ in Ref. \cite{XZ08}, $V=\left\{\left(V_{i}, v_{i}\right)\right\}_{i \in I}$ is a standard $(C,C')$-controlled $K$-fusion frame.

\end{proof}
\begin{cor}
Let $K\in End^*_{\am}(\h)$ be a closed range, and suppose that $W=\left\{\left(W_{i}, v_{i}\right)\right\}_{i \in I}$ is a standard $(C,C')$-controlled $K$-fusion frame with lower bound $A, a_{1}, a_{2}, a \in \mathfrak{A}$ with $\left\|a_{1}\right\|,\left\|a_{2}\right\|<1$ and $v=\left\{v_{i}\right\}_{i \in I} \in \ell^{2}(I, \am)$ is $a$ sequence of weights. If $V=\left\{\left(V_{i}, v_{i}\right)\right\}_{i \in I}$ satisfies
$$
\left|\left (C'^*(\pi_{W_{i}}-\pi_{V_{i}}\right)C f\right|^{2} \leq\left|a_{1} \right| C'^* \pi_{W_{i}} Cf\left|+a_{2}\right|  C'^* \pi_{V_{i}} C f|+| f|a|^{2}
$$
for each $f \in \h, i \in I$ and $\|a v\|_{2}<\left(1-\left\|a_{1}\right\|\right) \sqrt{A}$, then $V$ is a standard $(C,C')$-controlled $K$-fusion frame

\end{cor} 
\begin{proof}
Using (\ref{def: perturbation}) and Theorem ($(2.2 .6)$) in Ref. \cite{Mu90}, we get 
$$
v_{i}\left|\left( C'^*(\pi_{W_{i}}-\pi_{V_{i}}\right)C f\right|^{2} v_{i} \leq v_{i}\left|a_{1}\right|C'^* \pi_{W_{i}}C f\left|+a_{2}\right|C'^* \pi_{V_{i}} C f|+| f|a|^{2} v_{i}
$$
for each $f \in \h$ and $i \in I .$ Since $v_{i}$ 's are positive elements of $\mathcal{Z}(\am)$, the above inequality implies that
$$
\left|v_{i}C'^*\left(\pi_{W_{i}}-\pi_{V_{i}}\right) C f\right|^{2} \leq\left|a_{1}\right| v_{i}C'^* \pi_{W_{i}}C f\left|+a_{2}\right| v_{i} C'^* \pi_{V_{i}} C f|+| f\left|a v_{i}\right|^{2}
$$ $\left\|\left\{a v_{i}\right\}_{i \in I}\right\|_{l^2(\h)}<\left(1-\left\|a_{1}\right\|\right) \sqrt{A}$, the result follows from Theorem (\ref{thm: perturbation}).
\end{proof}
\bibliographystyle{amsplain}

\begin{thebibliography}{99}
 \bibitem{Ab05} I. Bogdanova, P. Vandergheynst, J.P. Antoine, L. Jacques, M. Morvidone: Stereo-graphic wavelet frames on the sphere, Applied Comput. Harmon. Anal. 19, (2005), 223-252.
 \bibitem{AK20} Assila, Nadia, Kabbaj, Samir and Moalige, Brahim. "Controlled -Fusion Frame for Hilbert Spaces" Moroccan Journal of Pure and Applied Analysis, vol.7, no.1, 2020, pp.116-133. https://doi.org/10.2478/mjpaa-2021-0011
\bibitem{AH13} M. Azhini, N. Haddadzadeh, Fusion frames in Hilbert modules over pro-$C^*$-algebras, Int. J. Industrial Mathematics (ISSN $2008-5621$), Vol. $5$, No. $2$, $2013$ Article ID IJIM-$00211$.

 \bibitem{BA10} P. Balaz, J-P. Antoine and A.Grybos, Wighted and controlled frames. Int. J. Wavelets Multi. Inf. Process., 8(10)(2010) 109-132.
 \bibitem{BAG10} P. Balazs, J. P. Antoine and A. Grybos, Weighted and controlled frames: mutual relationship and first
numerical properties, Int. J. Wavelets, Multiresolution Info. Proc., $14(2010)$, No. $1$, $109-132$.
\bibitem{BPY06} J. J. Benedetto, O. Yilmaz and A. M. Powell, "Sigma-delta quantization and finite frames," 2004 IEEE International Conference on Acoustics, Speech, and Signal Processing, 2004, pp. iii-937, doi: 10.1109/ICASSP.2004.1326700.
\bibitem{CA00} P. G. CASAZZA, The art of frame theory, Taiwanese J. Math. 4 (2000), 129–201.
\bibitem{CH03} O. CHRISTENSEN, An Introduction to Frames and Riesz Bases, Birkhauser, Boston, 2003.


\bibitem{CK08} P. Casazza and G. Kutyniok, Robustness of fusion frames under erasures of subspaces and of local frame vectors, Contemp. Math. 464 (2008) 149–160.
\bibitem{DGM86} I. DAUBECHIES, A. GROSSMANN, Y. MEYER, Painless nonorthogonal expansions, J. Math. Phys.
27 (1986), 1271–1283.
\bibitem{DS52} R. J. DUFFIN, A. C. SCHAEFFER, A class of nonharmonic Fourier series, Trans. Amer. Math. Soc.
72 (1952), 341–366.

\bibitem{EL95} E. Lance, Hilbert $C^*$-Modules—A Toolkit for Operator Algebraists, Cambridge University Press, $1995$.
\bibitem{FY09} X. Fang, J. Yu and H. Yao, Solutions to operator equations on Hilbert C-modules, Linear Alg. Appl., 431(11) (2009) 2142-2153.
\bibitem{KH08} Khosravi, A., B. Khosravi,"Fusion frames and G-frames in Hilbert C*-modules."  Int. J. International Journal of Wavelets, Multiresolution and Information ProcessingVol. 06, No. 03, pp. 433-446 (2008).
\bibitem{KS10} K.Sharifi Groetsch's representation of Moore–Penrose inverses and ill-posed problems in Hilbert C*-modules. Int. Journal of Mathematical Analysis and Applications. Volume 365, Issue 2, 15 May 2010, Pages 646-652

\bibitem{Mu90} G. J. Murphy, $C^*$-Algebras and Operator Theory (Academic Press, 1990).
\bibitem{NW93} N.E. Wegge-Olsen, $K$-Theory and $C*$-Algebras, Published in the United States  by Oxford University Press Inc., New York $1993$.

\bibitem{LA07} L. Arabasic, “On frames for contably generated Hilbert $C^*$-modules,”Proceedings of the American Mathematical Society, vol. 135, pp. 469–478, 2007.
\bibitem{LA95} Lance EC. Hilbert $C^*$-Modules A Toolkit for Operator Algebraist, London Math Soc Lecture Note Ser, Volume 210. Cambridge, UK: Cambridge University Press, $1995$.

 \bibitem{LR21} Hatim Labrigui, Abdeslam Touri, Mohamed Rossafi, Samir Kabbaj, "Controlled K-g-Operator Frames on Hilbert C$^*$-Modules", Journal of Mathematics, vol. 2021, Article ID 5530498, 8 pages, 2021. https://doi.org/10.1155/2021/5530498.
 \bibitem{LRT19} M. Rossafi, A. Touri, H. Labrigui, A. Akhlidj, "Continuous $\ast$-K-G-Frame in Hilbert $C^{\ast}$-Modules", Journal of Function Spaces, vol. 2019, Article ID 2426978, 5 pages, 2019. https://doi.org/10.1155/2019/2426978.

\bibitem{RK21} Mohamed Rossafi, Samir Kabbaj, "Some Generalizations of Frames in Hilbert Modules", International Journal of Mathematics and Mathematical Sciences, vol. 2021, Article ID 5522671, 11 pages, 2021. https://doi.org/10.1155/2021/5522671.

\bibitem{Ros2} M. Rossafi and S. Kabbaj, \emph{$\ast$-K-g-frames in Hilbert $C^{\ast}$-modules}, Journal of Linear and Topological Algebra Vol. 07, No. 01, 2018, 63- 71.

\bibitem{WP73} W. Paschke, “Inner product modules over $B*$-algebras,” Transactions of the American Mathematical Society, vol. 182, pp. 443–468, 1973.
 
\bibitem{XS08} Xu QX, Sheng LJ. Positive semi-definite matrices of adjointable operators on Hilbert $C^*$-modules.  Linear Algebra Appl 2008; 428: 992–1000. IEEE Trans. Inform. Theory 52 $(2006)$, $1990–2005$.
\bibitem{XZ08} X. Xiao and X. Zeng, Some properties of g-frames in Hilbert C$^*$-modules, J. Math. Anal. Appl. 363 (2010) 399–408.
\end{thebibliography}

\end{document}